\documentclass[11pt]{extarticle}
\usepackage[hidelinks, bookmarks=false]{hyperref}
\usepackage{amsmath, amsthm, amssymb}
\usepackage[shortlabels]{enumitem}
\usepackage[table]{xcolor}
\usepackage{graphicx}
\usepackage{subcaption}
\usepackage[all]{xypic}
\usepackage{makecell}
\usepackage{MnSymbol}
\usepackage{dsfont}
\usepackage{comment, float}
\usepackage{mathtools}
\usepackage{listings}
\usepackage{blkarray}
\usepackage{tikz}
\usepackage{float}
\usepackage{longtable}
\usepackage{bbm}
\usepackage[numbers]{natbib}
\usepackage{multirow}
\usepackage{mathrsfs} 

\oddsidemargin \evensidemargin
\usepackage{enumitem}
\setenumerate{leftmargin=*}
\allowdisplaybreaks[3]	

\newtheorem{theorem}{Theorem}
\numberwithin{theorem}{section}
\newtheorem{proposition}[theorem]{Proposition}
\newtheorem{lemma}[theorem]{Lemma}
\newtheorem{corollary}[theorem]{Corollary}
\newtheorem{definition}[theorem]{Definition}

\newtheorem{remark}[theorem]{Remark}
\newtheorem{example}[theorem]{Example}
\newtheorem{conjecture}[theorem]{Conjecture}

\newcommand{\NN}{\mathbb{N}}

\newcommand{\RR}{\mathbb{R}}

\newcommand{\alignqedhere}{\\[\dimexpr-\baselineskip+\dp\strutbox] &\qedhere}
\newcommand{\cross}{\rotatebox[origin=c]{45}{\normalsize $\square$}}

\usepackage{titlesec}
\titleformat{\section}{\normalfont\large\bfseries}{\thesection}{1em}{}
\titleformat{\subsection}{\normalfont\small\bfseries}{\thesubsection}{1em}{}

\date{}
\title{ \textbf{Likelihood Geometry of Reflexive Polytopes}}
\author{ Carlos Am\'endola, Janike Oldekop}

\begin{document}

\maketitle

\begin{abstract}
We study the problem of maximum likelihood (ML) estimation for statistical models defined by reflexive polytopes. Our focus is on the maximum likelihood degree of these models as an algebraic measure of complexity of the corresponding optimization problem. We compute the ML degrees of all 4319 classes of three-dimensional reflexive polytopes, and observe some surprising behavior in terms of the presence of gaps between ML degrees and degrees of the associated toric varieties. We interpret these drops in the context of discriminants and prove formulas for the ML degree for families of reflexive polytopes, including the hypercube and its dual, the cross polytope, in arbitrary dimension. In particular, we determine a family of embeddings for the $d$-cube that implies ML degree one. Finally, we discuss generalized constructions of families of reflexive polytopes in terms of their ML degrees. 
\end{abstract}

\section{Introduction}

The connection between \emph{log-linear models} and \emph{toric varieties} has been a cornerstone of algebraic statistics \citep{AlgebraicStatistics}. In this setting, the combinatorial and geometric object of a convex lattice polytope is naturally associated to a discrete exponential family. More precisely, given a lattice polytope, one can associate a log-linear model whose design matrix has the lattice points as columns. In this work we focus on a special kind of lattice polytopes known as
\emph{reflexive polytopes}.

Reflexive polytopes were originally introduced by \citet{DualPolyhedraAndMirrorSymmetryForCalabiYauHypersurfacesInToricVarieties} in the context of \emph{mirror symmetry}. This relationship between geometric objects called \emph{Calabi-Yau manifolds} is one of the most notable examples of the interplay between theoretical physics and algebraic geometry. More precisely, a reflexive polytope and its dual lead to a mirror-dual pair of Calabi-Yau manifolds \citep{MirrorSymmetryAndPolarDualityOfPolytopes}, which are of specific interest in string theory. In particular, six-dimensional Calabi-Yau manifolds are used for \emph{Kaluza-Klein compactification}. This connection led to a legitimate interest in the classification of reflexive polytopes by both mathematicians and physicists. Since each reflexive polytope has a unique interior lattice point, it is a well-known consequence of \citep{LatticeVertexPolytopesWithInteriorLatticePoints, BoundsForLatticePolytopesContainingAFixedNumberOfInteriorPointsInASublattice} that there are only finitely many equivalence classes of reflexive polytopes in each dimension, up to unimodular isomorphism. In two dimensions, all polygons with exactly one interior lattice point are reflexive, resulting in 16 isomorphism classes \citep[Theorem 4.2.3]{TheNumberOfModuliOfFamiliesOfCurvesOnToricSurfaces}. Although the number of equivalence classes grows very rapidly, the three-dimensional and four-dimensional reflexive polytopes were classified by \citet{OnTheClassificationOfReflexivePolyhedra, ClassificationOfReflexivePolyhedraInThreeDimensions,CompleteClassificationOfReflexivePolyhedraInFourDimensions}. 

In addition, reflexive polytopes exhibit various combinatorial properties \citep{1224AndBeyond}. Through the study of \emph{Gorenstein polytopes}, they are closely related to commutative algebra and combinatorics \citep{CombinatorialAspectsOfMirrorSymmetry, HVectorsOfGorensteinPolytopes, GorensteinToricFanoVarieties}. In toric geometry, reflexive polytopes have also been classified as important examples of \emph{Fano varieties}. And notably in algebraic statistics, these varieties have already appeared in the context of log-linear models associated to certain \emph{phylogenetic trees} \citep{OnGeometryOfBinarySymmetricModelsOfPhylogeneticTrees}.

The statistical setup we are addressing will be that of \emph{parameter inference} via the prominent method of \emph{maximum likelihood}. Analyzing this method has been an important theme in algebraic statistics, since the associated optimization problem can be examined using tools from algebra and geometry. In general, the ML estimator is determined by solving the \emph{score equations} \citep{SolvingTheLikelihoodEquations}. The number of complex solutions to these equations for sufficiently general data is known as the \emph{maximum likelihood degree} \citep{TheMaximumLikelihoodDegree}, and it serves as a measure of the algebraic complexity of the ML estimation problem. 

A further motivation for studying reflexive polytopes is that their faces are of special interest. \citet[Proposition 2.2]{TheReflexiveDimensionOfALatticePolytope} showed that every lattice polytope is isomorphic to a face of some reflexive polytope. From a likelihood geometry perspective, given a log-linear model the faces of the underlying polytope can be used to explain the ML degree via the \emph{principal $A$-determinant} \citep[Chapter 9]{DiscriminantsResultantsAndMultidimensionalDeterminants}. Fundamental to this is the correspondence between polytopes and toric varieties mentioned above. By studying embeddings of the toric variety given by different scalings, a reduction in complexity of ML estimation can be achieved. Indeed, the ML degree of a statistical model is at most its degree \citep{TheMaximumLikelihoodDegreeOfToricVarieties, LikelihoodGeometry}. This upper bound is achieved by toric varieties scaled by generic scalings, and the principal $A$-determinant determines the locus of non-generic scalings that exhibit a lower ML degree. Of particular interest in likelihood geometry is studying polytopes for which an ML degree one scaling can be determined \citep{VarietiesWithMaximumLikelihoodDegreeOne}. For the two and three dimensional cube it is known that such a scaling exists \citep{FamiliesOfPolytopesWithRationalLinearPrecisionInHigherDimensions}. The approach presented in this paper allows the specification of a family of ML degree one scalings for the $d$-dimensional cube. In this way we give an explicit example of a polytope and associated scalings in any dimension such that the embedded toric variety has ML degree one.

Log-linear models defined by reflexive polygons were examined with regard to their ML degree by \citet*{MaximumLikelihoodEstimationOfToricFanoVarieties}. In the present paper we revisit their computations, take the next step of considering (three-dimensional) reflexive polyhedra, and prove results that hold for families of reflexive polytopes in higher dimensions. While the results in \cite{MaximumLikelihoodEstimationOfToricFanoVarieties} might suggest at first glance that an ML degree drop for a reflexive polytope is a rare occurrence - only 12.5\% of reflexive polygons exhibit such a drop - we actually find that for reflexive polyhedra more than 64\% of these exhibit an ML degree drop. Moreover, we show that in arbitrary dimension there are always reflexive polytopes that exhibit an ML degree drop (with their standard embedding).

The rest of the paper is structured as follows. In Section \ref{section2} we recall preliminaries on maximum likelihood estimation of log-linear models. The focus is on the relationship between statistical models defined by reflexive polytopes and the theory of $A$-discriminants based on \citep{MaximumLikelihoodEstimationOfToricFanoVarieties, TheMaximumLikelihoodDegreeOfToricVarieties}. Section \ref{section3} revisits the computation of the ML degree of all reflexive polygons, and corrects Theorem 3.1 in \citep{MaximumLikelihoodEstimationOfToricFanoVarieties}. Computations and results related to the ML degrees of three-dimensional reflexive polytopes are presented in Section \ref{section4}. In Section \ref{section5} we  present closed formulas for the ML degree of the cube and its dual in arbitrary dimension, and discuss scalings that make the former have ML degree one. We focus on reflexive simplices in Section \ref{section6}, and conjecture formulas for their ML degrees. The ML degree of further families of reflexive polytopes is studied in the last two sections. In Section \ref{section7} we look at geometric constructions that create higher dimensional reflexive polytopes from lower dimensional ones, and in Section \ref{section8} we consider reflexive polytopes associated to undirected graphs.

All our computations are available and reproducible at the mathematical research data repository \texttt{MathRepo} of the Max-Planck Institute of Mathematics in the Sciences, with webpage:\begin{center}
    \url{https://mathrepo.mis.mpg.de/LikelihoodReflexive}
\end{center}

\section{Log-Linear Models and Maximum Likelihood Estimation} \label{section2}

In this article we consider a specific class of statistical models known as log-linear models. These arise naturally from an algebraic study of discrete regular exponential families \citep[Chapter 6.2]{AlgebraicStatistics}. Let $[n] \coloneqq \{ 1, 2, \dots, n\}$ be a nonempty finite set. Consider 
\begin{equation*}
\Delta_{n-1} \coloneqq \{ (p_1, \ldots, p_n) \in \mathbb{R}^n \mid p_1, \ldots, p_n \ge 0 \textup{ and } p_1 + \ldots + p_n = 1 \}
\end{equation*}
as the set of probability measures on $[n]$, known as the \emph{probability simplex}. Throughout the article, $A \in \mathbb{Z}^{d \times n}$ is a matrix of the form $A = \begin{bmatrix}
a_1 & a_2 & \ldots & a_n \\
\end{bmatrix}$, where $a_1, \ldots, a_n \in \mathbb{Z}^d$, 
and we let 
\begin{equation*}
A' = \begin{bmatrix}
1 & 1 & \ldots & 1 \\
a_1 & a_2 & \ldots & a_n \\
\end{bmatrix}.
\end{equation*}

\begin{definition}\label{def:loglinear}
Let $A  \in \mathbb{Z}^{d\times n}$. The \emph{log-linear model} associated to $A$ is 
\begin{equation*}
\mathcal{M}_A \coloneqq \{ p \in \Delta_{n-1} \mid \textup{log } p \in \textup{rowspan}(A') \}.
\end{equation*}
The matrix $A$ is called \emph{design matrix}.
\end{definition}

As is customary in algebraic statistics, we added the all-ones vector as a first row in $A'$ to ensure that the uniform distribution is included in the log-linear model $\mathcal{M}_A$. This is additionally motivated by interpreting $\mathcal{M}_A$ as a discrete exponential family, and the associated toric ideal will be homogeneous \cite[Section 6.2]{AlgebraicStatistics}.

From a geometric perspective, we consider a log-linear model as a toric variety intersected with the probability simplex $\Delta_{n-1}$. For this purpose we write $\theta^{a_j} \coloneqq \theta_1^{a_{1j}} \dots \theta_d^{a_{d j}}$ for the monomial in $\theta$ defined by $a_j$ occurring in $A$. For each $c \in (\mathbb{C}^*)^n$ we define the parametrization map
\begin{equation} \label{gl1}
\psi_A^c : (\mathbb{C}^*)^{d+1} \to (\mathbb{C}^*)^n, \quad (s, \theta_1, \ldots, \theta_d) \mapsto (c_1 s \theta^{a_1}, \ldots, c_n s \theta^{a_n}).
\end{equation}

\begin{definition} \label{definition2}
Let $A \in \mathbb{Z}^{d \times n}$ and $c \in (\mathbb{C}^*)^n$. The \emph{scaled toric variety} $V_A^c \subseteq \mathbb{C}^n$ is the Zariski closure of $\textup{im}\, \psi_A^c$. We refer to $c = (1,1,\ldots,1)$ as the \emph{standard scaling}, and $V_A \coloneqq V_A^{(1,1,\ldots,1)}$ is the \emph{toric variety} associated to $A$.
\end{definition}

In this way, each scaling gives rise to a parametrization corresponding to a different embedding of $V_A$. We omit $A$ from the notation whenever it is clear from the context. If we do not specify a scaling explicitly, we mean the standard one.

For log-linear models, the maximum likelihood estimation problem simplifies to the following characterization, known as \emph{Birch's theorem}. We assume that observed \emph{data} is given by $u = (u_1, \ldots, u_n) \in \mathbb{N}^n$, where $u_i$ describes the number of observations of the $i$-th event. Let $u_+ \coloneqq u_1 + \ldots + u_n$ be the associated \emph{sample size}.

\begin{theorem}\cite[Corollary 7.3.9]{AlgebraicStatistics} \label{theorem3}
Let $A \in \mathbb{Z}^{d \times n}$. The \emph{maximum likelihood estimate (MLE)} over the model $\mathcal{M}_A$ for given data $u \in \mathbb{N}^n$ is the unique solution $\hat{p}$, if it exists, to
\begin{equation*}
\frac{1}{u_+} A u = A p \quad \text{and} \quad p \in \mathcal{M}_A.
\end{equation*}
\end{theorem}

We often work with the parametrization given by $\psi^c_A$ in \eqref{gl1}, so that the MLE $\hat{p}$ is the image of some $\hat{\theta}$. The equations appearing in Theorem \ref{theorem3} are called \emph{score equations}. An essential invariant for the ML estimation problem is given by the maximum likelihood degree \citep{TheMaximumLikelihoodDegree}, defined as follows.

\begin{definition}
Given $A \in \mathbb{Z}^{d \times n}$ and $c\in (C^\ast)^n$, the \emph{ML degree} of $V_A^c$, denoted by $\textup{mldeg}(V_A^c)$, is the number of complex solutions to the score equations for generic data $u \in \mathbb{N}^n$. 
\end{definition}

Our primary interest is in a particular class of log-linear models. In many cases the columns of $A$ are given by lattice points contained in a polytope \citep{MaximumLikelihoodEstimationInLogLinearModels, LikelihoodInferenceInExponentialFamiliesAndDirectionsOfRecession, OnTheGeometryOfDiscreteExponentialFamiliesWithApplicationToExponentialRandomGraphModels}. We refer to \citet{LecturesOnPolytopes} for a detailed introduction to polytopes. A special subclass of polytopes are the reflexive ones. 

\begin{definition}
A lattice polytope is \emph{reflexive} if it contains the origin in its interior and its dual polytope is also a lattice polytope.
\end{definition}

Modeling using reflexive polytopes allows the use of tools from discrete geometry to study the ML degree. In the following we establish the relationship between $A$-discriminants introduced by \citet*[Chapter 9]{DiscriminantsResultantsAndMultidimensionalDeterminants} and associated ML degrees. This idea goes back to the work of the Mathematics Research Communities (MRC) \emph{Likelihood Geometry} group \cite{TheMaximumLikelihoodDegreeOfToricVarieties}. See also \citep{LikelihoodGeometry} for a detailed introduction to aspects of likelihood geometry. 

Let $P \subseteq \mathbb{R}^d$ be a $d$-dimensional reflexive polytope with lattice points $a_1, \ldots, a_n$ defining $A$. From the underlying polytope $P$ we can deduce the following geometric information on $V$. 

\begin{theorem} \citep[Theorem 4.16]{sturmfels1996grobner} \label{theorem6} 
Let $A = [a_1, \ldots, a_n] \in \mathbb{Z}^{d \times n}$. The degree of $V_A$ is the normalized volume $d! \cdot \textup{vol}(P)$ of $P = \textup{conv}(a_1, \ldots, a_n)$.
\end{theorem} 

\noindent According to \citep[Theorem 3.2]{LikelihoodGeometry} and \citep[Corollary 8]{TheMaximumLikelihoodDegreeOfToricVarieties}, the ML degree is bounded by the degree:
\begin{equation} \label{gl2}
\textup{mldeg} (V^c) \le \textup{deg} (V) \quad \textup{for all } c \in (\mathbb{C}^*)^n.
\end{equation}
Equality holds for generic $c$. This motivates the following definition.

\begin{definition}
For all $c \in (\mathbb{C}^*)^n$, the \emph{ML degree drop} of $V^c$ is $\textup{mldrop} (V^c) \coloneqq \textup{deg} (V^c) - \textup{mldeg} (V)$.
\end{definition}

\noindent Based on $A \in \mathbb{Z}^{d \times n}$, each scaling $c \in (\mathbb{C}^*)^n$ induces a polynomial $f_c = \sum_{i = 1}^n c_i \theta^{a_i}$.

\begin{definition}
For any matrix $A$ as above, the \emph{variety of scalings} is defined as
\begin{equation*}
\nabla_A \coloneqq \overline{\left \{ c \in (\mathbb{C}^*)^n \mid \exists \theta \in (\mathbb{C}^*)^d \text{ such that } f_c (\theta) = \frac{\partial f_c}{\partial \theta_i} (\theta) = 0 \text{ for all } i \right \}}.
\end{equation*}
If $\nabla_A$ has codimension one in $(\mathbb{C}^*)^n$, then the \emph{$A$-discriminant} $\Delta_A$ is defined to be the irreducible polynomial that vanishes on $\nabla_A$.
\end{definition}

\noindent Define the \emph{principal $A$-determinant} as 
\begin{equation*}
E_A(c) \coloneqq \prod_{\Gamma \text{ face of } P} \Delta_{\Gamma \cap A} (c),
\end{equation*}
where the product is taken over all nonempty faces $\Gamma \subseteq P$ including $P$ itself and $\Gamma \cap A$ is the matrix whose columns correspond to the lattice points contained in $\Gamma$. This polynomial is an important object to decide on an ML degree drop. See \citep{TheMaximumLikelihoodDegree, LikelihoodGeometry} for more information on the relationship between singularities and ML degree drops.

\begin{theorem}\label{thm:Adet} \cite[Theorem 2]{TheMaximumLikelihoodDegreeOfToricVarieties}
Let $c \in (\mathbb{C}^*)^n$ be a fixed scaling and consider the scaled toric variety $V^c$. Then $\textup{mldeg}(V^c) < \textup{deg} (V)$ if and only if $E_A (c) = 0$.
\end{theorem}

\section{Reflexive Polygons} \label{section3}

Reflexive polygons are classified by 16 isomorphism classes \citep[Proposition 4.1]{GorensteinToricFanoVarieties} shown in Figure~\ref{figure1}. Maximum likelihood estimates of the associated models have been studied in \cite{MaximumLikelihoodEstimationOfToricFanoVarieties}. We make a correction to \citep[Theorem 3.1]{MaximumLikelihoodEstimationOfToricFanoVarieties}, by clarifying that not only the reflexive polygon $P_{5\textup{a}}$ exhibits an ML degree drop (its degree is 5 but its ML degree is 3), but also the 
reflexive polygon $P_{8\textup{a}}$ exhibits a drop. The other 14 polygons do not. Our computations using \texttt{Macaulay2} \citep{Macaulay2} and the \texttt{AlgebraicOptimization.m2} package \citep{AlgebraicOptimizationDegree} can be verified in the \texttt{MathRepo} page.

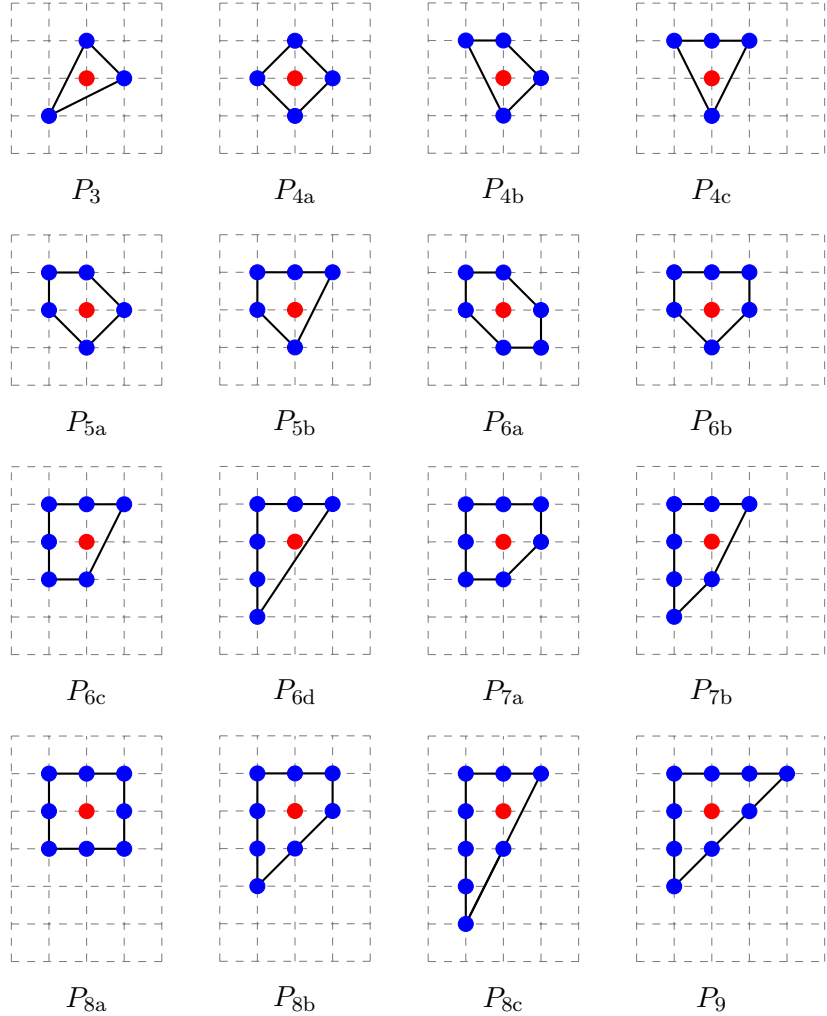
\begin{figure}[h!]
\centering
\begin{tikzpicture} 
\draw [step=0.5,gray, dashed] (0,0) grid (2.0,2.0);

\draw[thick] (0.5, 0.5) -- (1.0, 1.5) {};
\draw[thick] (0.5, 0.5) -- (1.5, 1.0) {};
\draw[thick] (1.5, 1.0) -- (1.0, 1.5) {};

\node[draw, circle, inner sep=2pt, fill, blue] at (0.5, 0.5) {};
\node[draw, circle, inner sep=2pt, fill, blue] at (1.0, 1.5) {};
\node[draw, circle, inner sep=2pt, fill, blue] at (1.5, 1.0) {};
\node[draw, circle, inner sep=2pt, fill, red] at (1.0, 1.0) {};
\node at (1.0, -0.5) {$P_3$};
\end{tikzpicture}
\hspace{0.5cm}
\begin{tikzpicture} 
\draw [step=0.5,gray, dashed] (0,0) grid (2.0,2.0);

\draw[thick] (0.5, 1.0) -- (1.0, 1.5) {};
\draw[thick] (1.0, 0.5) -- (1.5, 1.0) {};
\draw[thick] (1.5, 1.0) -- (1.0, 1.5) {};
\draw[thick] (0.5, 1.0) -- (1.0, 0.5) {};

\node[draw, circle, inner sep=2pt, fill, blue] at (1.0, 0.5) {};
\node[draw, circle, inner sep=2pt, fill, blue] at (1.0, 1.5) {};
\node[draw, circle, inner sep=2pt, fill, blue] at (1.5, 1.0) {};
\node[draw, circle, inner sep=2pt, fill, blue] at (0.5, 1.0) {};
\node[draw, circle, inner sep=2pt, fill, red] at (1.0, 1.0) {};
\node at (1.0, -0.5) {$P_{4\textup{a}}$};
\end{tikzpicture}
\hspace{0.5cm}
\begin{tikzpicture} 
\draw [step=0.5,gray, dashed] (0,0) grid (2.0,2.0);

\draw[thick] (0.5, 1.5) -- (1.0, 1.5) {};
\draw[thick] (1.0, 0.5) -- (1.5, 1.0) {};
\draw[thick] (1.5, 1.0) -- (1.0, 1.5) {};
\draw[thick] (0.5, 1.5) -- (1.0, 0.5) {};

\node[draw, circle, inner sep=2pt, fill, blue] at (1.0, 0.5) {};
\node[draw, circle, inner sep=2pt, fill, blue] at (1.0, 1.5) {};
\node[draw, circle, inner sep=2pt, fill, blue] at (1.5, 1.0) {};
\node[draw, circle, inner sep=2pt, fill, blue] at (0.5, 1.5) {};
\node[draw, circle, inner sep=2pt, fill, red] at (1.0, 1.0) {};
\node at (1.0, -0.5) {$P_{4\textup{b}}$};
\end{tikzpicture}
\hspace{0.5cm}
\begin{tikzpicture} 
\draw [step=0.5,gray, dashed] (0,0) grid (2.5,2.0);

\draw[thick] (0.5, 1.5) -- (1.0, 1.5) {};
\draw[thick] (1.0, 0.5) -- (1.5, 1.5) {};
\draw[thick] (1.5, 1.5) -- (1.0, 1.5) {};
\draw[thick] (0.5, 1.5) -- (1.0, 0.5) {};

\node[draw, circle, inner sep=2pt, fill, blue] at (1.0, 0.5) {};
\node[draw, circle, inner sep=2pt, fill, blue] at (1.0, 1.5) {};
\node[draw, circle, inner sep=2pt, fill, blue] at (1.5, 1.5) {};
\node[draw, circle, inner sep=2pt, fill, blue] at (0.5, 1.5) {};
\node[draw, circle, inner sep=2pt, fill, red] at (1.0, 1.0) {};
\node at (1.0, -0.5) {$P_{4\textup{c}}$};
\end{tikzpicture}

\vspace{0.25cm}

\begin{tikzpicture} 
\draw [step=0.5,gray, dashed] (0,0) grid (2.0,2.0);

\draw[thick] (0.5, 1.5) -- (1.0, 1.5) {};
\draw[thick] (1.0, 0.5) -- (1.5, 1.0) {};
\draw[thick] (1.5, 1.0) -- (1.0, 1.5) {};
\draw[thick] (0.5, 1.0) -- (1.0, 0.5) {};
\draw[thick] (0.5, 1.5) -- (0.5, 1.0) {};

\node[draw, circle, inner sep=2pt, fill, blue] at (1.0, 0.5) {};
\node[draw, circle, inner sep=2pt, fill, blue] at (1.0, 1.5) {};
\node[draw, circle, inner sep=2pt, fill, blue] at (1.5, 1.0) {};
\node[draw, circle, inner sep=2pt, fill, blue] at (0.5, 1.5) {};
\node[draw, circle, inner sep=2pt, fill, blue] at (0.5, 1.0) {};
\node[draw, circle, inner sep=2pt, fill, red] at (1.0, 1.0) {};
\node at (1.0, -0.5) {$P_{5\textup{a}}$};
\end{tikzpicture}
\hspace{0.5cm}
\begin{tikzpicture} 
\draw [step=0.5,gray, dashed] (0,0) grid (2.0,2.0);

\draw[thick] (0.5, 1.5) -- (1.0, 1.5) {};
\draw[thick] (1.0, 0.5) -- (1.5, 1.5) {};
\draw[thick] (1.5, 1.5) -- (1.0, 1.5) {};
\draw[thick] (0.5, 1.0) -- (1.0, 0.5) {};
\draw[thick] (0.5, 1.5) -- (0.5, 1.0) {};

\node[draw, circle, inner sep=2pt, fill, blue] at (1.0, 0.5) {};
\node[draw, circle, inner sep=2pt, fill, blue] at (1.0, 1.5) {};
\node[draw, circle, inner sep=2pt, fill, blue] at (1.5, 1.5) {};
\node[draw, circle, inner sep=2pt, fill, blue] at (0.5, 1.5) {};
\node[draw, circle, inner sep=2pt, fill, blue] at (0.5, 1.0) {};
\node[draw, circle, inner sep=2pt, fill, red] at (1.0, 1.0) {};
\node at (1.0, -0.5) {$P_{5\textup{b}}$};
\end{tikzpicture}
\hspace{0.5cm}
\begin{tikzpicture} 
\draw [step=0.5,gray, dashed] (0,0) grid (2.0,2.0);

\draw[thick] (0.5, 1.5) -- (1.0, 1.5) {};
\draw[thick] (1.0, 0.5) -- (1.5, 0.5) {};
\draw[thick] (1.5, 1.0) -- (1.0, 1.5) {};
\draw[thick] (0.5, 1.0) -- (1.0, 0.5) {};
\draw[thick] (0.5, 1.5) -- (0.5, 1.0) {};
\draw[thick] (1.5, 0.5) -- (1.5, 1.0) {};

\node[draw, circle, inner sep=2pt, fill, blue] at (1.0, 0.5) {};
\node[draw, circle, inner sep=2pt, fill, blue] at (1.0, 1.5) {};
\node[draw, circle, inner sep=2pt, fill, blue] at (1.5, 1.0) {};
\node[draw, circle, inner sep=2pt, fill, blue] at (0.5, 1.5) {};
\node[draw, circle, inner sep=2pt, fill, blue] at (0.5, 1.0) {};
\node[draw, circle, inner sep=2pt, fill, blue] at (1.5, 0.5) {};
\node[draw, circle, inner sep=2pt, fill, red] at (1.0, 1.0) {};
\node at (1.0, -0.5) {$P_{6\textup{a}}$};
\end{tikzpicture}
\hspace{0.5cm}
\begin{tikzpicture} 
\draw [step=0.5,gray, dashed] (0,0) grid (2.5,2.0);

\draw[thick] (0.5, 1.5) -- (1.0, 1.5) {};
\draw[thick] (1.0, 0.5) -- (1.5, 1.0) {};
\draw[thick] (1.5, 1.5) -- (1.0, 1.5) {};
\draw[thick] (0.5, 1.0) -- (1.0, 0.5) {};
\draw[thick] (0.5, 1.5) -- (0.5, 1.0) {};
\draw[thick] (1.5, 1.5) -- (1.5, 1.0) {};

\node[draw, circle, inner sep=2pt, fill, blue] at (1.0, 0.5) {};
\node[draw, circle, inner sep=2pt, fill, blue] at (1.0, 1.5) {};
\node[draw, circle, inner sep=2pt, fill, blue] at (1.5, 1.0) {};
\node[draw, circle, inner sep=2pt, fill, blue] at (0.5, 1.5) {};
\node[draw, circle, inner sep=2pt, fill, blue] at (0.5, 1.0) {};
\node[draw, circle, inner sep=2pt, fill, blue] at (1.5, 1.5) {};
\node[draw, circle, inner sep=2pt, fill, red] at (1.0, 1.0) {};
\node at (1.0, -0.5) {$P_{6\textup{b}}$};
\end{tikzpicture}

\vspace{0.25cm}

\begin{tikzpicture} 
\draw [step=0.5,gray, dashed] (0,-0.5) grid (2.0,2.0);

\draw[thick] (0.5, 1.5) -- (1.0, 1.5) {};
\draw[thick] (1.0, 0.5) -- (0.5, 0.5) {};
\draw[thick] (1.5, 1.5) -- (1.0, 1.5) {};
\draw[thick] (0.5, 1.0) -- (0.5, 0.5) {};
\draw[thick] (0.5, 1.5) -- (0.5, 1.0) {};
\draw[thick] (1.5, 1.5) -- (1.0, 0.5) {};

\node[draw, circle, inner sep=2pt, fill, blue] at (1.0, 0.5) {};
\node[draw, circle, inner sep=2pt, fill, blue] at (1.0, 1.5) {};
\node[draw, circle, inner sep=2pt, fill, blue] at (0.5, 0.5) {};
\node[draw, circle, inner sep=2pt, fill, blue] at (0.5, 1.5) {};
\node[draw, circle, inner sep=2pt, fill, blue] at (0.5, 1.0) {};
\node[draw, circle, inner sep=2pt, fill, blue] at (1.5, 1.5) {};
\node[draw, circle, inner sep=2pt, fill, red] at (1.0, 1.0) {};
\node at (1.0, -1.0) {$P_{6\textup{c}}$};
\end{tikzpicture}
\hspace{0.5cm}
\begin{tikzpicture} 
\draw [step=0.5,gray, dashed] (0,-0.5) grid (2.0,2.0);

\draw[thick] (0.5, 1.5) -- (1.0, 1.5) {};
\draw[thick] (0.5, 0.0) -- (0.5, 0.5) {};
\draw[thick] (1.5, 1.5) -- (1.0, 1.5) {};
\draw[thick] (0.5, 1.0) -- (0.5, 0.5) {};
\draw[thick] (0.5, 1.5) -- (0.5, 1.0) {};
\draw[thick] (1.5, 1.5) -- (0.5, 0.0) {};

\node[draw, circle, inner sep=2pt, fill, blue] at (0.5, 0.0) {};
\node[draw, circle, inner sep=2pt, fill, blue] at (1.0, 1.5) {};
\node[draw, circle, inner sep=2pt, fill, blue] at (0.5, 0.5) {};
\node[draw, circle, inner sep=2pt, fill, blue] at (0.5, 1.5) {};
\node[draw, circle, inner sep=2pt, fill, blue] at (0.5, 1.0) {};
\node[draw, circle, inner sep=2pt, fill, blue] at (1.5, 1.5) {};
\node[draw, circle, inner sep=2pt, fill, red] at (1.0, 1.0) {};
\node at (1.0, -1.0) {$P_{6\textup{d}}$};
\end{tikzpicture}
\hspace{0.5cm}
\begin{tikzpicture} 
\draw [step=0.5,gray, dashed] (0,-0.5) grid (2.0,2.0);

\draw[thick] (0.5, 1.5) -- (1.0, 1.5) {};
\draw[thick] (1.0, 0.5) -- (0.5, 0.5) {};
\draw[thick] (1.5, 1.5) -- (1.0, 1.5) {};
\draw[thick] (0.5, 1.0) -- (0.5, 0.5) {};
\draw[thick] (0.5, 1.5) -- (0.5, 1.0) {};
\draw[thick] (1.5, 1.5) -- (1.5, 1.0) {};
\draw[thick] (1.0, 0.5) -- (1.5, 1.0) {};

\node[draw, circle, inner sep=2pt, fill, blue] at (1.0, 0.5) {};
\node[draw, circle, inner sep=2pt, fill, blue] at (1.0, 1.5) {};
\node[draw, circle, inner sep=2pt, fill, blue] at (0.5, 0.5) {};
\node[draw, circle, inner sep=2pt, fill, blue] at (0.5, 1.5) {};
\node[draw, circle, inner sep=2pt, fill, blue] at (0.5, 1.0) {};
\node[draw, circle, inner sep=2pt, fill, blue] at (1.5, 1.5) {};
\node[draw, circle, inner sep=2pt, fill, blue] at (1.5, 1.0) {};
\node[draw, circle, inner sep=2pt, fill, red] at (1.0, 1.0) {};
\node at (1.0, -1.0) {$P_{7\textup{a}}$};
\end{tikzpicture}
\hspace{0.5cm}
\begin{tikzpicture} 
\draw [step=0.5,gray, dashed] (0,-0.5) grid (2.5,2.0);

\draw[thick] (0.5, 1.5) -- (1.0, 1.5) {};
\draw[thick] (0.5, 0.0) -- (0.5, 0.5) {};
\draw[thick] (1.5, 1.5) -- (1.0, 1.5) {};
\draw[thick] (0.5, 1.0) -- (0.5, 0.5) {};
\draw[thick] (0.5, 1.5) -- (0.5, 1.0) {};
\draw[thick] (1.5, 1.5) -- (1.0, 0.5) {};
\draw[thick] (1.0, 0.5) -- (0.5, 0.0) {};

\node[draw, circle, inner sep=2pt, fill, blue] at (1.0, 0.5) {};
\node[draw, circle, inner sep=2pt, fill, blue] at (1.0, 1.5) {};
\node[draw, circle, inner sep=2pt, fill, blue] at (0.5, 0.5) {};
\node[draw, circle, inner sep=2pt, fill, blue] at (0.5, 1.5) {};
\node[draw, circle, inner sep=2pt, fill, blue] at (0.5, 1.0) {};
\node[draw, circle, inner sep=2pt, fill, blue] at (1.5, 1.5) {};
\node[draw, circle, inner sep=2pt, fill, blue] at (0.5, 0.0) {};
\node[draw, circle, inner sep=2pt, fill, red] at (1.0, 1.0) {};
\node at (1.0, -1.0) {$P_{7\textup{b}}$};
\end{tikzpicture}

\vspace{0.25cm}

\begin{tikzpicture} 
\draw [step=0.5,gray, dashed] (0,-1.0) grid (2.0,2.0);

\draw[thick] (0.5, 1.5) -- (1.0, 1.5) {};
\draw[thick] (1.0, 0.5) -- (0.5, 0.5) {};
\draw[thick] (1.5, 1.5) -- (1.0, 1.5) {};
\draw[thick] (0.5, 1.0) -- (0.5, 0.5) {};
\draw[thick] (0.5, 1.5) -- (0.5, 1.0) {};
\draw[thick] (1.5, 1.5) -- (1.5, 0.5) {};
\draw[thick] (1.0, 0.5) -- (1.5, 0.5) {};

\node[draw, circle, inner sep=2pt, fill, blue] at (1.0, 0.5) {};
\node[draw, circle, inner sep=2pt, fill, blue] at (1.0, 1.5) {};
\node[draw, circle, inner sep=2pt, fill, blue] at (0.5, 0.5) {};
\node[draw, circle, inner sep=2pt, fill, blue] at (0.5, 1.5) {};
\node[draw, circle, inner sep=2pt, fill, blue] at (0.5, 1.0) {};
\node[draw, circle, inner sep=2pt, fill, blue] at (1.5, 1.5) {};
\node[draw, circle, inner sep=2pt, fill, blue] at (1.5, 0.5) {};
\node[draw, circle, inner sep=2pt, fill, blue] at (1.5, 1.0) {};
\node[draw, circle, inner sep=2pt, fill, red] at (1.0, 1.0) {};
\node at (1.0, -1.5) {$P_{8\textup{a}}$};
\end{tikzpicture}
\hspace{0.5cm}
\begin{tikzpicture} 
\draw [step=0.5,gray, dashed] (0,-1.0) grid (2.0,2.0);

\draw[thick] (0.5, 1.5) -- (1.0, 1.5) {};
\draw[thick] (0.5, 0.0) -- (0.5, 0.5) {};
\draw[thick] (1.5, 1.5) -- (1.0, 1.5) {};
\draw[thick] (0.5, 1.0) -- (0.5, 0.5) {};
\draw[thick] (0.5, 1.5) -- (0.5, 1.0) {};
\draw[thick] (1.5, 1.5) -- (1.5, 1.0) {};
\draw[thick] (1.0, 0.5) -- (0.5, 0.0) {};
\draw[thick] (1.0, 0.5) -- (1.5, 1.0) {};

\node[draw, circle, inner sep=2pt, fill, blue] at (1.0, 0.5) {};
\node[draw, circle, inner sep=2pt, fill, blue] at (1.0, 1.5) {};
\node[draw, circle, inner sep=2pt, fill, blue] at (0.5, 0.5) {};
\node[draw, circle, inner sep=2pt, fill, blue] at (0.5, 1.5) {};
\node[draw, circle, inner sep=2pt, fill, blue] at (0.5, 1.0) {};
\node[draw, circle, inner sep=2pt, fill, blue] at (1.5, 1.5) {};
\node[draw, circle, inner sep=2pt, fill, blue] at (0.5, 0.0) {};
\node[draw, circle, inner sep=2pt, fill, blue] at (1.5, 1.0) {};
\node[draw, circle, inner sep=2pt, fill, red] at (1.0, 1.0) {};
\node at (1.0, -1.5) {$P_{8\textup{b}}$};
\end{tikzpicture}
\hspace{0.5cm}
\begin{tikzpicture} 
\draw [step=0.5,gray, dashed] (0,-1.0) grid (2.0,2.0);

\draw[thick] (0.5, 1.5) -- (1.0, 1.5) {};
\draw[thick] (0.5, 0.0) -- (0.5, 0.5) {};
\draw[thick] (1.5, 1.5) -- (1.0, 1.5) {};
\draw[thick] (0.5, 1.0) -- (0.5, 0.5) {};
\draw[thick] (0.5, 1.5) -- (0.5, 1.0) {};
\draw[thick] (1.5, 1.5) -- (1.0, 0.5) {};
\draw[thick] (0.5, -0.5) -- (0.5, 0.0) {};
\draw[thick] (1.0, 0.5) -- (0.5, -0.5) {};
\draw[thick] (1.0, 0.5) -- (0.5, -0.5) {};

\node[draw, circle, inner sep=2pt, fill, blue] at (1.0, 0.5) {};
\node[draw, circle, inner sep=2pt, fill, blue] at (1.0, 1.5) {};
\node[draw, circle, inner sep=2pt, fill, blue] at (0.5, 0.5) {};
\node[draw, circle, inner sep=2pt, fill, blue] at (0.5, 1.5) {};
\node[draw, circle, inner sep=2pt, fill, blue] at (0.5, 1.0) {};
\node[draw, circle, inner sep=2pt, fill, blue] at (1.5, 1.5) {};
\node[draw, circle, inner sep=2pt, fill, blue] at (0.5, 0.0) {};
\node[draw, circle, inner sep=2pt, fill, blue] at (0.5, -0.5) {};
\node[draw, circle, inner sep=2pt, fill, red] at (1.0, 1.0) {};
\node at (1.0, -1.5) {$P_{8\textup{c}}$};
\end{tikzpicture}
\hspace{0.5cm}
\begin{tikzpicture} 
\draw [step=0.5,gray, dashed] (0,-1.0) grid (2.5,2.0);

\draw[thick] (0.5, 1.5) -- (1.0, 1.5) {};
\draw[thick] (0.5, 0.0) -- (0.5, 0.5) {};
\draw[thick] (1.5, 1.5) -- (1.0, 1.5) {};
\draw[thick] (0.5, 1.0) -- (0.5, 0.5) {};
\draw[thick] (0.5, 1.5) -- (0.5, 1.0) {};
\draw[thick] (1.5, 1.5) -- (2.0, 1.5) {};
\draw[thick] (1.0, 0.5) -- (2.0, 1.5) {};
\draw[thick] (1.0, 0.5) -- (0.5, 0.0) {};

\node[draw, circle, inner sep=2pt, fill, blue] at (1.0, 0.5) {};
\node[draw, circle, inner sep=2pt, fill, blue] at (1.0, 1.5) {};
\node[draw, circle, inner sep=2pt, fill, blue] at (0.5, 0.5) {};
\node[draw, circle, inner sep=2pt, fill, blue] at (0.5, 1.5) {};
\node[draw, circle, inner sep=2pt, fill, blue] at (0.5, 1.0) {};
\node[draw, circle, inner sep=2pt, fill, blue] at (1.5, 1.5) {};
\node[draw, circle, inner sep=2pt, fill, blue] at (0.5, 0.0) {};
\node[draw, circle, inner sep=2pt, fill, blue] at (2.0, 1.5) {};
\node[draw, circle, inner sep=2pt, fill, blue] at (1.5, 1.0) {};
\node[draw, circle, inner sep=2pt, fill, red] at (1.0, 1.0) {};
\node at (1.0, -1.5) {$P_{9}$};
\end{tikzpicture}

\caption{The 16 isomorphism classes of reflexive polygons. Each label indicates the name of the corresponding polygon, with a label number that counts the lattice points on the boundary.} \label{figure1}
\end{figure}

\begin{proposition} \label{proposition10}
The log-linear model defined by $P_{8\textup{a}}$ exhibits an ML degree drop. Concretely, $\textup{deg}(V) = 8$ and $\textup{mldeg}(V) = 4$.
\end{proposition}

\begin{proof}
After an affine transformation, the lattice points of $P_{8\textup{a}}$ define the matrix
\begin{equation*}
A = \begin{bmatrix}
0 & 1 & 2 & 0 & 1 & 2 & 0 & 1 & 2 \\
2 & 2 & 2 & 1 & 1 & 1 & 0 & 0 & 0 \\
\end{bmatrix}.
\end{equation*}
By Theorem \ref{theorem6}, $\textup{deg} (V) = 2! \cdot \textup{vol} (P_{8\textup{a}}) = 8$.

To show $\textup{mldeg}(V) = 4$, we study $E_A(c)$. Vertices $a_i$ have $\Delta_{a_i} = 1$. The upper edge $e=\mathrm{conv}\{(0,2),(2,2)\}$ defines the polynomial $f_{c,e} = c_{02} \theta_2^2+c_{12}\theta_1 \theta_2^2+c_{22}\theta_1^2 \theta_2^2$. It has a non\-tri\-vi\-al singularity if and only if $c_{02}+c_{12}\theta_1+c_{22}\theta_1^2$ does. Consequently, $\Delta_{e}(c) = c_{12}^2-4c_{02}c_{22} \neq 0$ for $c = (1, 1, \dots, 1)$. The remaining three edges can be treated analogously. Therefore, it is sufficient to study $\nabla_A$. The corresponding polynomial is \begin{align*}
    f &= \theta_2^2+\theta_1 \theta_2^2+\theta_1^2 \theta_2^2+\theta_2+\theta_1 \theta_2+\theta_1^2 \theta_2+1+\theta_1+\theta_1^2 \\
    &= (\theta_1^2+\theta_1+1)(\theta_2^2+\theta_2+1) 
\end{align*} 
and therefore has four singularities $(\theta_1,\theta_2)$ given by the solutions of $\theta_1^2+\theta_1+1=\theta_2^2+\theta_2+1=0$. 
It follows that there is indeed an ML degree drop and moreover $\textup{mldeg}(V) = 8 - 4 = 4$.
\end{proof}

This example will be generalized in higher dimensions in Section \ref{section5}. In particular, an explicit form for the MLE is given in Proposition \ref{prop:critcube}.

\section{Reflexive Polyhedra} \label{section4}

While the classification of reflexive polygons is relatively straightforward, the classification of reflexive polytopes in higher dimensions is a challenging problem. It is theoretically possible due to the classification algorithm of \citet{OnTheClassificationOfReflexivePolyhedra}. An application of this algorithm to the three-dimensional case can be studied in \cite{ClassificationOfReflexivePolyhedraInThreeDimensions}. The corresponding 4\,319 polytopes are listed in the \texttt{KS} database which can be found at \url{http://hep.itp.tuwien.ac.at/~kreuzer/CY/CYk3.html}. We refer to polytope $i$, by which we mean the $i$-th entry in the database, numbering starts at 0. Our main interest is the degree of the associated toric variety for each polytope, its maximum likelihood degree and their relation. 

\begin{figure}[h]
\centering
\begin{subfigure}[b]{0.49\textwidth}
(a)
\end{subfigure}
\hfill
\begin{subfigure}[b]{0.49\textwidth}
(b)
\end{subfigure}
\centering
\begin{subfigure}[b]{0.49\textwidth}
\centering
\includegraphics[width=\textwidth]{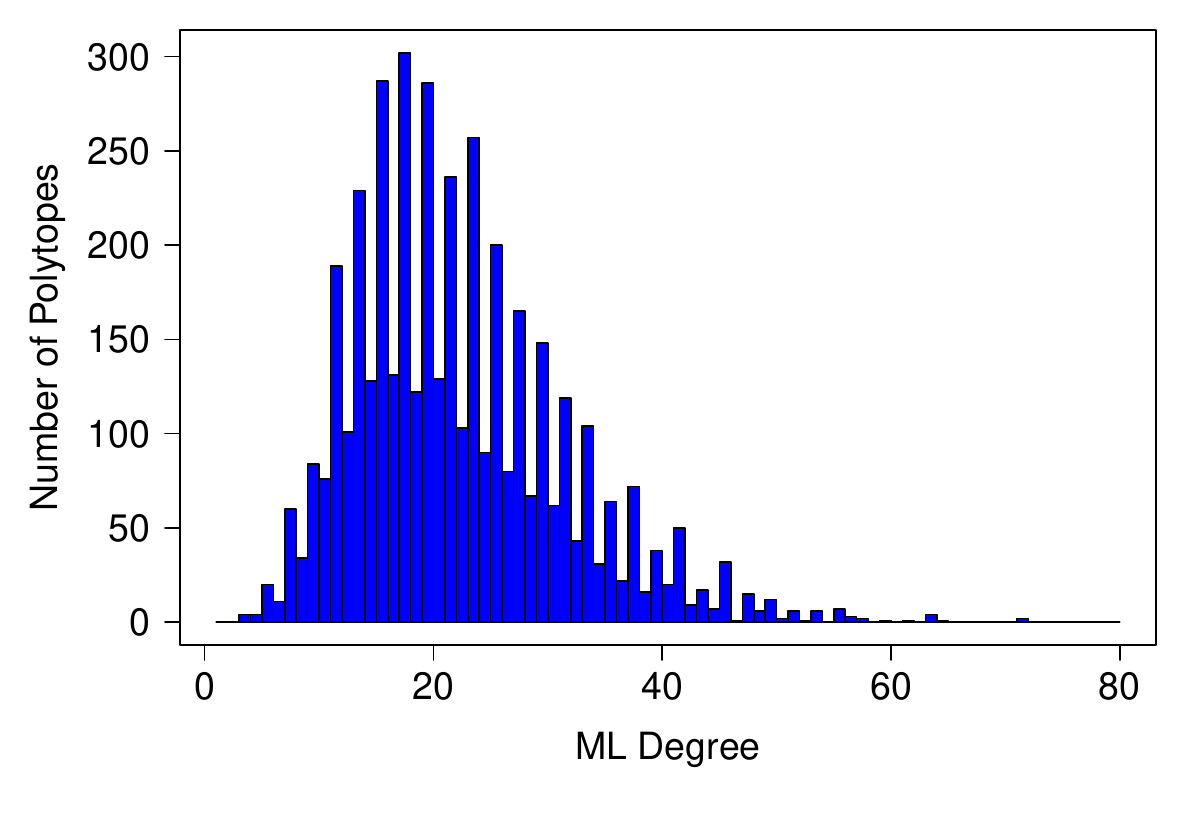}
\end{subfigure}
\hfill
\begin{subfigure}[b]{0.49\textwidth}
\centering
\includegraphics[width=\textwidth]{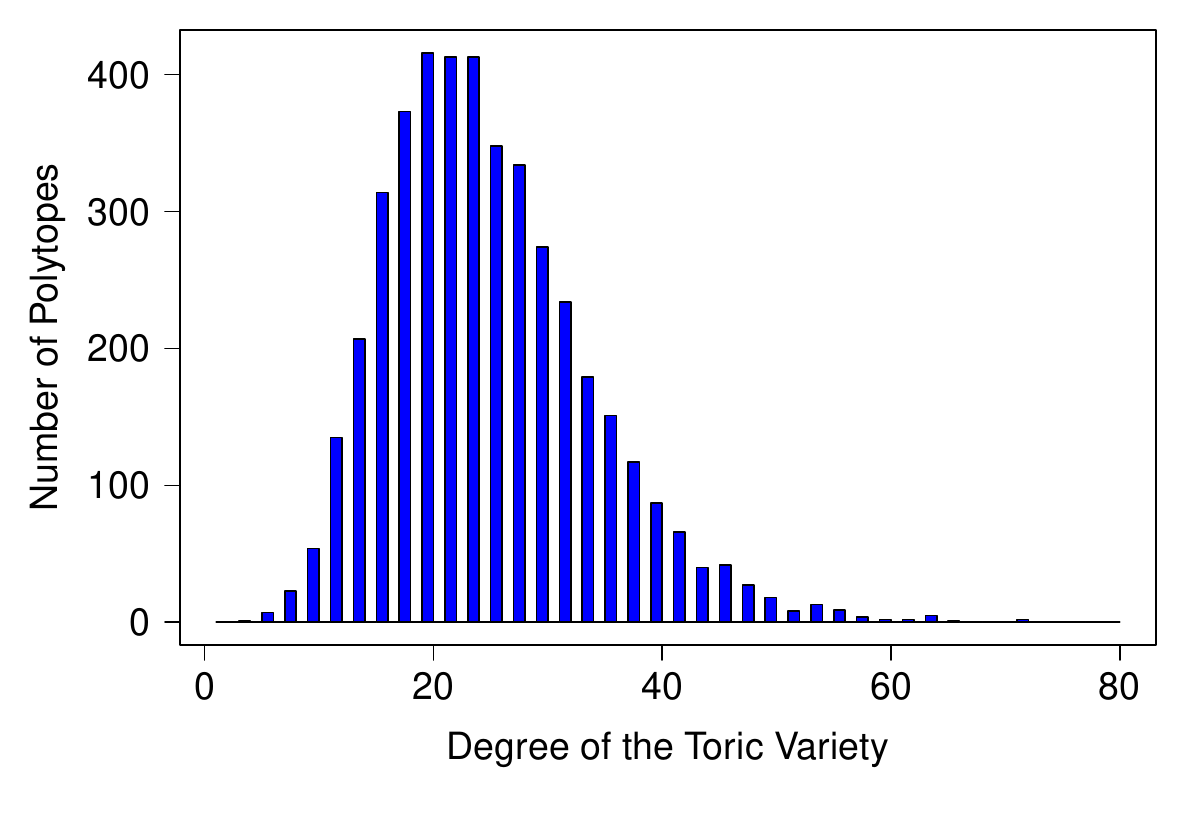}
\end{subfigure}
\begin{subfigure}[c]{1\textwidth}
\hspace{4.2cm}(c)
\end{subfigure}
\hspace{0.475\textwidth}
\begin{subfigure}[c]{\textwidth}
\centering
\includegraphics[width=0.49\textwidth]{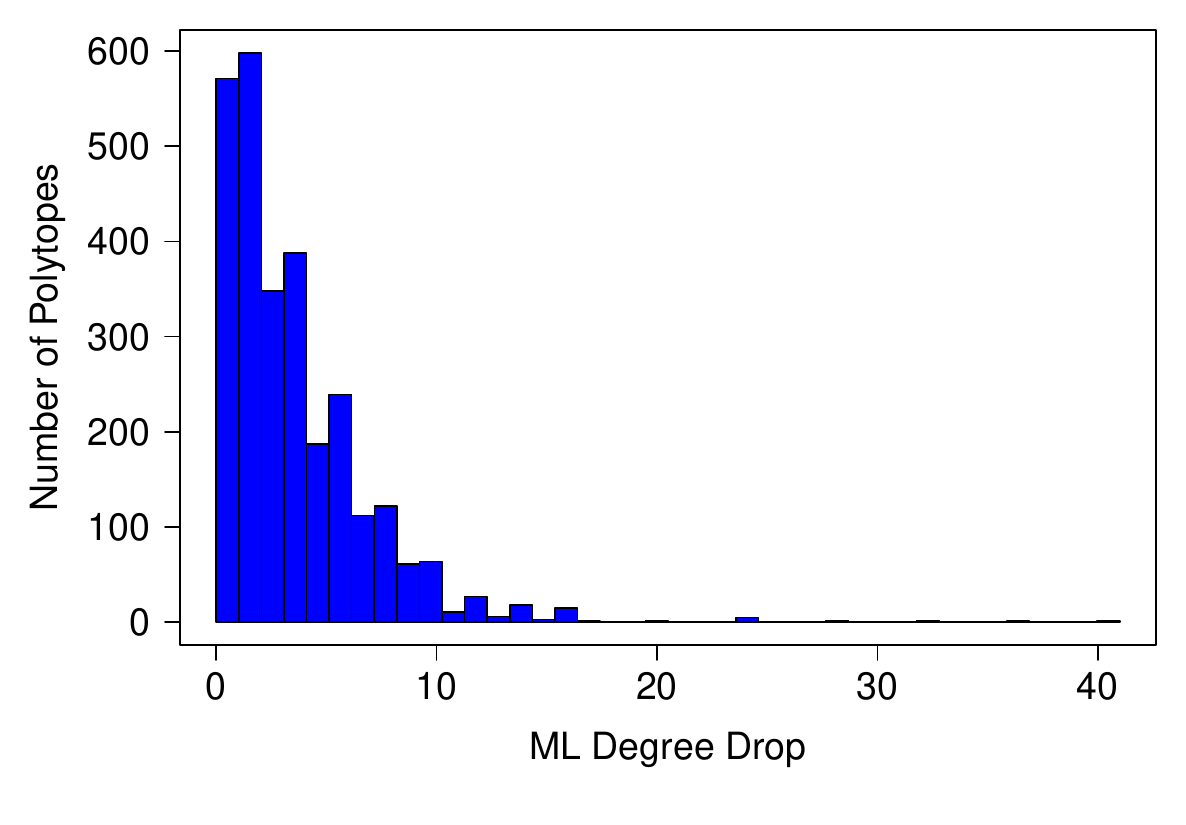}
\end{subfigure}
\hfill
\hspace{0.49\textwidth}
\caption[]{\small Histograms of the data obtained by computations with three-dimensional reflexive polytopes. In (a) the frequency distribution of $\textup{mldeg}(V)$ is shown, while in (b) the frequency distribution of $\textup{deg}(V)$ is illustrated. Histogram (c) shows the frequency distribution of all ML degree drops larger than zero.} \label{figure2}
\end{figure} 

Using \texttt{Macaulay2} \citep{Macaulay2}, the package \texttt{ReflexivePolytopesDB.m2} allows access to the mentioned database. The computer-aided analysis of the models requires the consideration of polynomials instead of Laurent polynomials. Therefore we consider all polytopes under translation such that all contained lattice points are non-negative. All computations were performed using the standard scaling. We computed $\textup{deg} (V)$ for each reflexive polytope using \texttt{FourTiTwo.m2}, an \texttt{Macaulay2} interface to most functions of the software \texttt{4ti2} \citep{4ti2}. We also determined the number of generators of the toric ideal. The main tool we used for both steps is the \texttt{toricMarkov} function. For the computation of all ML degrees we used homotopy continuation \citep{SolvingPolynomialSystemsViaHomotopyContinuationAndMonodromy, CoefficientParameterPolynomialContinuation}, in particular the \texttt{Julia} \citep{JuliaAFreshApproachToNumericalComputing} package \texttt{HomotopyContinuation.jl} \citep{HomotopyContinuationJL}. Further invariants determined using the \texttt{Polyhedra.m2} package are the components of the $f$-vector and the number of lattice points. The $f$-vector of a three-dimensional polytope is of the form $(f_0, f_1, f_2)$, where $f_i$ is the number of $i$-dimensional faces contained in the polytope.

In order to visualize the results, the frequency distributions of ML degrees and degrees of the toric varieties are represented by histograms in Figure \ref{figure2}. One observes that Histogram (b) presents gaps in the possible values for the degree, and this can be explained by the following proposition.

\begin{proposition}
The toric variety associated to an odd-dimensional reflexive polytope has even degree. 
\end{proposition}

\begin{proof}
This is a consequence of results in Ehrhart theory  \citep[Chapter 3]{ComputingTheContinuousDiscretely}. Let $d=2k+1$,  and consider the $h^\star$-vector $(h_0^\star, h_1^\star, \dots, h_d^\star) \in \mathbb{N}^{d+1}$. 
A classical result states that $$\textup{nvol} (P) = h_0^\star + h_1^\star + h_2^\star + \dots + h_d^\star.$$ 
Now, Hibi's palindromic theorem \citep{EhrhartPolynomialsOfConvexPolytopesHVectorsOfSimplicialComplexesAndNonsingularProjectiveToricVarieties} states that a polytope is reflexive if and only if its $h^*$-vector is palindromic. So we have
$(h_0^\ast, h_1^\ast, \dots, h_{k}^\ast, h_{k}^\ast, \dots, h_1^\ast, h_0^\ast)$. 
Therefore, by Theorem \ref{theorem6},
\begin{equation*}
\textup{deg} (V) = \textup{nvol} (P) = 2 \sum_{i=0}^k h_i^\star. \qedhere
\end{equation*}
\end{proof}

Among the 4\,319 reflexive polytopes, we find that 2\,784 exhibit an ML degree drop. The maximum ML degree drop is given by 40. The only three-dimensional reflexive polytope exhibiting an ML degree drop 40 is the $3$-cube corresponding to polytope \texttt{418}. Figure \ref{figure2} shows the histogram of the frequency distribution of all ML degree drops larger than zero.

The minimum ML degree is given by four. This corresponds to four models whose data sets are presented in Table \ref{table1}. Three of these models exhibit an ML degree drop. The corresponding polytopes are shown in Figure \ref{figure3}. 

\begin{table}[h!]
\centering
\begin{tabular}{c c c c c c c c} 
Polytope & mldeg($V$) & deg($V$) & $f_0$ & $f_1$ & $f_2$ & \# Lattice Points & \# Generators \\ [0.2ex] 
\hline \\ [-1ex] 
 \texttt{0} & 4 & 4 & 4 & 6 & 4 & 5 & 1 \\ 
 \texttt{1184} & 4 & 10 & 7 & 12 & 7 & 8 & 9 \\
 \texttt{2582} & 4 & 12 & 8 & 14 & 8 & 9 & 10 \\
 \texttt{4101} & 4 & 18 & 11 & 19 & 10 & 12 & 28 \\ [-1ex] 
\end{tabular}
\caption{Reflexive polyhedra of ML degree four.} \label{table1}
\end{table}

\begin{figure}[h!]
\centering
\begin{subfigure}[b]{0.475\textwidth}
\centering
\includegraphics[scale=0.5]{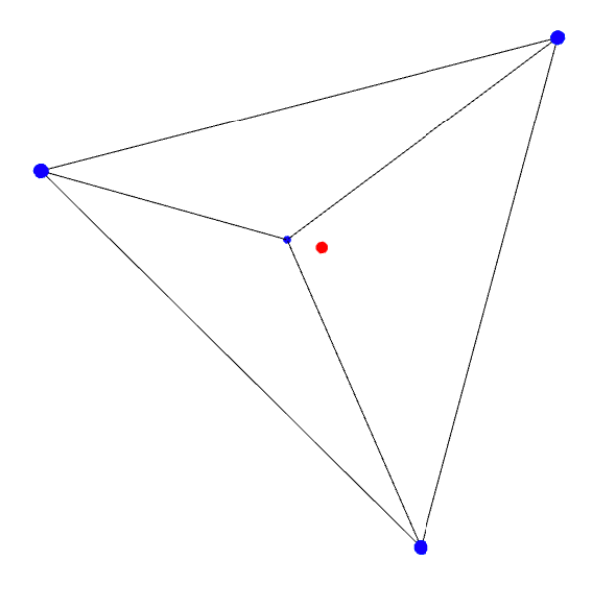}
\end{subfigure}
\hfill
\begin{subfigure}[b]{0.475\textwidth}
\centering
\includegraphics[scale=0.5]{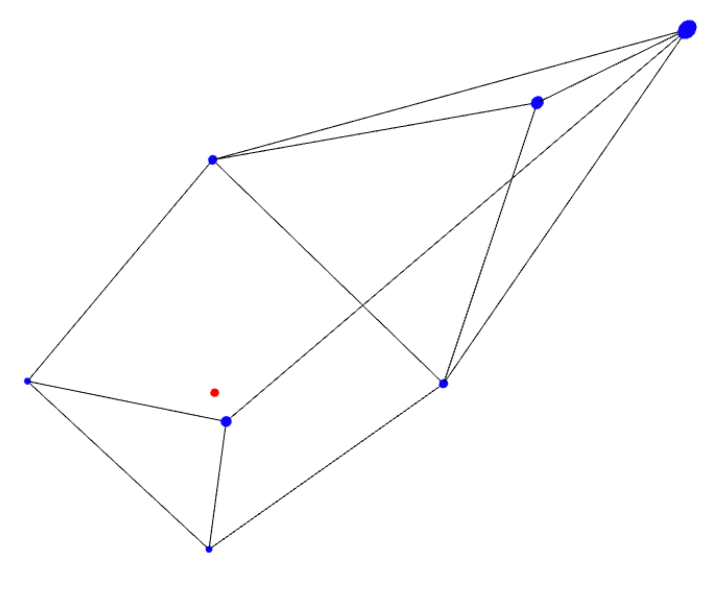}
\end{subfigure}
\centering
\begin{subfigure}[b]{0.475\textwidth}
(a) Polytope \texttt{0}
\end{subfigure}
\hfill
\begin{subfigure}[b]{0.475\textwidth}
(b) Polytope \texttt{1184}
\end{subfigure}
\vskip\baselineskip
\begin{subfigure}[b]{0.475\textwidth}
\centering
\includegraphics[scale=0.5]{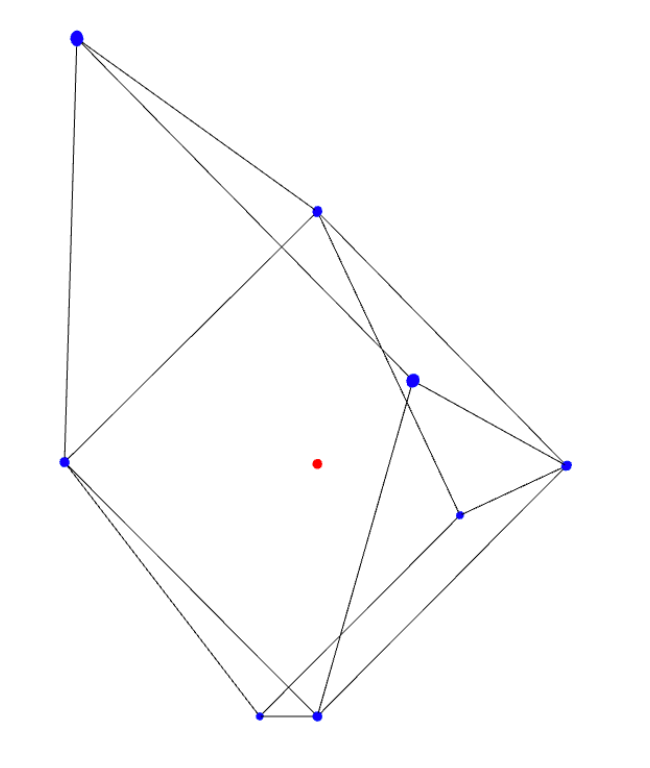}
\end{subfigure}
\hfill
\begin{subfigure}[b]{0.475\textwidth}
\centering
\includegraphics[scale=0.5]{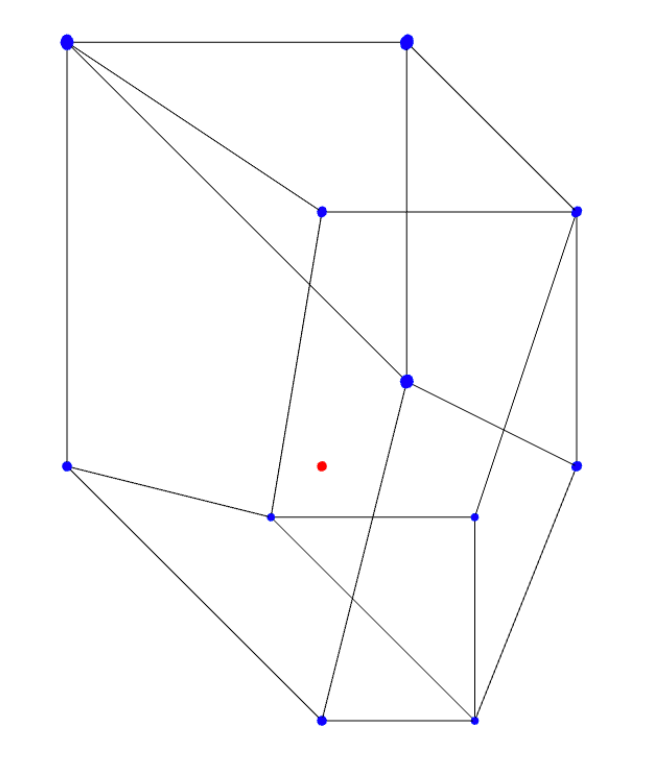}
\end{subfigure}
\centering
\begin{subfigure}[b]{0.475\textwidth}
(c) Polytope \texttt{2582}
\end{subfigure}
\hfill
\begin{subfigure}[b]{0.475\textwidth}
(d) Polytope \texttt{4101}
\end{subfigure}
\caption[]{\small Three-dimensional reflexive polytopes with ML degree four for the standard scaling. The red point corresponds to the interior lattice point. The visualizations were created using Polymake \citep{ComputingConvexHullsAndCountingIntegerPointsWithPolymake, polymakeAFrameworkForAnalyzingConvexPolytopes}.} \label{figure3}
\end{figure}

While Polytope \texttt{0} did not exhibit an ML degree drop, according to Theorem \ref{thm:Adet}, a scaling can change this. We illustrate this situation with the following example.

\begin{example} \upshape
After an affine transformation, the design matrix for Polytope \texttt{0} is given by
\begin{equation*}
A = \begin{blockarray}{ccccc}
a_0 & a_1 & a_2 & a_3 & a_4\\
\begin{block}{[ccccc]}
1 & 2 & 1 & 1 & 0\\
1 & 1 & 2 & 1 & 0\\
1 & 1 & 1 & 2 & 0\\
\end{block}
\end{blockarray}.
\end{equation*}
The proper faces are $a_1, a_2, a_3, a_4$ and
\begin{alignat*}{2}
e_0 &= \mathrm{conv}(a_1, a_2), \qquad &&e_1 = \mathrm{conv}(a_1, a_3), \qquad e_2 = \mathrm{conv}(a_1, a_4),  \\
e_3 &= \mathrm{conv}(a_2, a_3), \qquad &&e_4 = \mathrm{conv}(a_2, a_4), \qquad e_5 = \mathrm{conv}(a_3, a_4), 
\end{alignat*}
$$
\Gamma_0 = \mathrm{conv}(a_1, a_2, a_3), \qquad  \Gamma_1 = \mathrm{conv}(a_1, a_2, a_4), \qquad
\Gamma_2 = \mathrm{conv}(a_1, a_3, a_4), \qquad  \Gamma_3 = \mathrm{conv}(a_2, a_3, a_4).
$$
One sees that the variety of scalings is empty for all these proper faces. The remaining discriminant $\Delta_A$ can be computed with the following \texttt{Macaulay2} code.
\begin{verbatim}
    R = QQ[c_111,c_211,c_121,c_112,c_000,t_1,t_2,t_3]
    f = c_111*t_1*t_2*t_3 + c_211*t_1^2*t_2*t_3 + c_121*t_1*t_2^2*t_3 +
          c_112*t_1*t_2*t_3^2 + c_000
    I = ideal(f,diff(t_1,f),diff(t_2,f),diff(t_3,f))
    J = ideal(t_1*t_2*t_3)
    K = saturate(I,J)
    eliminate({t_1,t_2,t_3},K)
\end{verbatim}
The result is the ideal generated by $\Delta_A = c_{111}^4-256 c_{211} c_{121} c_{112} c_{000}$. 

Using the scaling $c_{111} = 4, c_{211} = c_{121} = c_{112} = c_{000} = 1$, we find that the ML degree is three. \end{example}

Histogram (a) in Figure \ref{figure2} shows that there are as many ML degree five polyhedra as there are ML degree four polyhedra. The associated data is shown in Table \ref{table2}.

\begin{table}[h!]
\centering
\begin{tabular}{c c c c c c c c} 
Polytope & mldeg($V$) & deg($V$) & $f_0$ & $f_1$ & $f_2$ & \# Lattice Points & \# Generators \\ [0.2ex] 
\hline \\ [-1ex] 
\texttt{132} & 5 & 6 & 5 & 8 & 5 & 6 & 2 \\ 
\texttt{379} & 5 & 8 & 6 & 11 & 7 & 7 & 5 \\
\texttt{3314} & 5 & 14 & 9 & 16 & 9 & 10 & 15 \\
\texttt{3778} & 5 & 14 & 9 & 16 & 9 & 10 & 15 \\ [-1ex] 
\end{tabular}
\caption{Reflexive polyhedra of ML degree five.} \label{table2}
\end{table}

We observe that now all four polytopes exhibit an ML degree drop with the standard scaling. The next example shows how by modifying the scaling we can make the drop even larger. 

\begin{example} \upshape
After an affine transformation, the design matrix of Polytope \texttt{132} is given by
\begin{equation*}
A = \begin{blockarray}{cccccc}
a_0 & a_1 & a_2 & a_3 & a_4 & a_5\\
\begin{block}{[cccccc]}
1 & 2 & 1 & 0 & 0 & 1\\
1 & 1 & 1 & 1 & 0 & 2\\
1 & 1 & 2 & 0 & 1 & 1\\
\end{block}
\end{blockarray}.
\end{equation*}
The proper faces are $a_1, a_2, \ldots, a_5$ and 
\begin{alignat*}{2}
e_0 &= \mathrm{conv}(a_1, a_2), \qquad e_1 = \mathrm{conv}(a_1, a_3), \qquad e_2 &&= \mathrm{conv}(a_1, a_4), \qquad e_3 = \mathrm{conv}(a_1, a_5), \\
e_4 &= \mathrm{conv}(a_2, a_4), \qquad e_5 = \mathrm{conv}(a_2, a_5), \qquad e_6 &&= \mathrm{conv}(a_3, a_4), \qquad e_7 = \mathrm{conv}(a_3, a_5), 
\end{alignat*}
\vspace{-0.7cm}
$$ \Gamma_1 = \mathrm{conv}(a_1, a_3, a_4), \qquad
\Gamma_2 = \mathrm{conv}(a_1, a_3, a_5), \qquad \Gamma_3 = \mathrm{conv}(a_1, a_2, a_5), \qquad 
\Gamma_4 = \mathrm{conv}(a_1, a_2, a_4),$$
$$ 
    \Gamma_0 = \mathrm{conv}(a_2, a_3, a_4, a_5). 
$$
The variety of scalings is empty for all proper faces except for the last one $\Gamma_0$. For $c = (1,1,\ldots,1)$,
\begin{equation*}
\Delta_{\Gamma_0} = c_{112} c_{010} - c_{001} c_{121} = 0,
\end{equation*}
i.e., $\Gamma_0$ explains the ML degree drop shown in Table \ref{table2}. Considering the whole polytope,
\begin{align*}
\Delta_A &= c_{111}^6 + 54 c_{111}^3 c_{211} c_{112} c_{010}+ 729 c_{211}^2 c_{112}^2 c_{010}^2 + 54 c_{111}^3 c_{211} c_{001} c_{121} \\
&\quad -1458 c_{211}^2 c_{112} c_{010} c_{001} c_{121} + 729 c_{211}^2 c_{001}^2 c_{121}^2.
\end{align*}
For $c_{111} = c_{112} = c_{010} = c_{001} = c_{121} = 1$ and $c_{211} = -\frac{1}{108}$ we have that both $\Delta_{\Gamma_0} = \Delta_A = 0$. We verified that with this choice of scaling the ML degree drops further to four. \end{example}

We believe this database of ML degrees for all reflexive polyhedra to be a good starting point for further study. One can look for polytopes with special properties of interest, for example, those that are \emph{smooth}. The classification of smooth polytopes has been discussed in many ways \citep{SmoothFanoPolytopesWithManyVertices, ClassifyingSmoothLatticePolytopesViaToricFibrations, AClassificationOfSmoothConvex3PolytopesWithAtMost16LatticePoints}. In the reflexive case, all smooth polytopes up to dimension nine are known. They were determined using \O bro's SFP algorithm \citep{AnAlgorithmForTheClassificationOfSmoothFanoPolytopes} and are listed, for example, in the \texttt{polyDB} database \citep{polyDBADatabaseForPolytopesAndRelatedObjects} (up to lattice equivalence). We could not find any information in the literature about which polytope numbers in the \texttt{KS} database correspond to smooth polytopes. Therefore, we determined the numbers using the smooth reflexive polytopes listed in the \texttt{polyDB} database and the \texttt{LatticePolytopes.m2} package. Table \ref{table3} shows the correspondences including ML degrees and degrees of the toric variety. Note that only two of the 18 smooth reflexive polytopes in three dimensions exhibit no ML degree drop, namely Polytope \texttt{1} and Polytope \texttt{129}. In other words, we see that conditioning on the property of being smooth increases the frequency of a reflexive polyhedron exhibiting an ML degree drop to almost $89\%$.

\begin{table}[h!]
\centering
\begin{tabular}{c c c c } 
\makecell{Polytope number in the \\ \texttt{KS} database} & \makecell{Polytope number in the \\ \texttt{polyDB} database} & mldeg(V) & deg(V) \\ [0.2ex] 
\hline \\ [-1ex] 
 1 & 18 & 64 & 64 \\ 
 127 & 17 & 18 & 54 \\
 129 & 15 & 56 &56 \\
 135 & 14 & 51 & 54 \\
 235 & 2 & 56 & 62 \\
 418 & 16 & 8 & 48 \\
 484 & 13 & 38 & 44 \\
 486 & 11 & 43 & 46 \\
 490 & 1 & 49 & 50 \\
 492 & 6 & 42 & 52 \\
 496 & 7 & 44 & 50 \\
 510 & 12 & 16 & 48 \\
 1152 & 9 & 14 & 42 \\
 1377 & 8 & 31 & 40 \\
 1379 & 5 & 29 & 44 \\
 1383 & 3 & 45 & 46 \\
 2310 & 10 & 12 & 36 \\
 2627 & 4 & 32 & 36 \\ [1ex] 
\end{tabular}
\caption{The 18 smooth polytopes in the \texttt{KS} database and their correspondence to the collection of smooth reflexive polytopes in the \texttt{polyDB} database in three dimensions. In the \texttt{KS} database, numbering starts at 0. The third and fourth columns show the associated (ML) degree.} \label{table3}
\end{table}

\section{Hypercubes and Cross Polytopes} \label{section5}

One of the best-known examples of a reflexive polytope in arbitrary dimension is the $d$-dimensional hypercube $C_d = [-1,1]^d$, also known as the $d$-cube. According to (\ref{gl2}), an upper bound for the ML degree of $C_d$ is given by the degree of the toric variety defined by $C_d$.

\begin{proposition} \label{proposition12}
The degree of $C_d$ is $d! \cdot 2^d$.
\end{proposition}

\begin{proof}
By Theorem \ref{theorem6}, it is enough to compute the normalized volume of $C_d$. One way is using the $f$-vector $(f_0, f_1, \ldots, f_{d-1})$ of $C_d$ and the fact that $f_{d-1} = 2d$, so that the recursive formula $\textup{deg}(C_d) = \textup{deg}(C_{d-1}) f_{d-1}$ holds. Alternatively, the symmetry group of $C_d$ is a realization of the hyperoctahedral group $(\mathbb{Z}/2)^d \rtimes S_d$ \citep{ACharacterRelationshipBetweenSymmetricGroupAndHyperoctahedralGroup}, whose order is given by $d! \cdot 2^d$, see e.g. \citep{AttractorsWithTheSymmetryOfTheNCube}.
\end{proof}

An inductive argument shows that the degree of $C_d$ is a strict upper bound on the ML degree.

\begin{corollary}
For $d>1$, $C_d$ exhibits an ML degree drop. 
\end{corollary}

\begin{proof}
For $d=2$ the statement was proved in Proposition \ref{proposition10}, since polygon $P_{8\textup{a}}$ is precisely a square. By induction, we assume that the statement holds for $C_d$. The $(d+1)$-cube contains $2(d+1)$ cubes of dimension $d$ in its boundary, and hence their $A$-discriminants appear as a factor in $E_A(c)$ where $A$ is the design matrix of $C_{d+1}$.
\end{proof}

Of course, one wishes to determine a closed formula for the ML degree of $C_d$. We do not keep the reader in suspense and immediately give the answer.

\begin{theorem} \label{thm:mldegcube}
The ML degree of the $d$-cube $C_d$ is $2^d$.
\end{theorem}

The quickest way to see this is to appeal to \cite[Theorem 5.5]{MaximumLikelihoodEstimationOfToricFanoVarieties} which states that the ML degree is multiplicative on a \emph{toric fiber product} \cite{sullivant2007toric}, of which the Cartesian product is a special case. We will proceed to state and give a self-contained proof of such weaker form of the theorem. This not only may be appreciated by the reader, but is also useful to us writing down the explicit score equations that we will refer to later.

\begin{corollary}\label{cor:product}(of \cite[Theorem 5.5]{MaximumLikelihoodEstimationOfToricFanoVarieties})
Let $P$ and $Q$ be lattice polytopes. Then
\begin{equation*}
\textup{mldeg} (P \times Q) = \textup{mldeg} (P) \cdot \textup{mldeg} (Q). 
\end{equation*}
\end{corollary} 

\begin{proof}
Let $P$ be a $d_1$-dimensional lattice polytope with design matrix $P = [p_1, \ldots, p_n] \in \mathbb{Z}^{d_1 \times n}$ and let $Q$ be a $d_2$-dimensional lattice polytope with design matrix $Q = [q_1, \ldots, q_m] \in \mathbb{Z}^{d_2 \times m}$. Here, denote $P_k$ and $Q_k$ the $k$-th row of $P$ and $Q$, respectively. Set $b_P = \frac{1}{u_+} Pu$ and $b_Q = \frac{1}{u_+} Qu$. Consider 
\begin{align*}
\psi_P : (\mathbb{C}^*)^{d_1+1} \to (\mathbb{C}^*)^n, \quad \theta_P \mapsto (s \theta_P^{p_1}, \ldots, s \theta_P^{p_n}), \\
\psi_Q : (\mathbb{C}^*)^{d_2+1} \to (\mathbb{C}^*)^m, \quad \theta_Q \mapsto (s \theta_Q^{q_1}, \ldots, s \theta_Q^{q_m}).
\end{align*}

The score equations of $P$ are
$$
s \sum_{i=1}^n \theta_P^{p_i} = 1, \qquad
s \cdot P_k \cdot (\theta_P^{p_1}, \ldots, \theta_P^{p_n}) = b_{P,k}, \quad k \in [d_1],
$$
and have $\textup{mldeg}(P)$ solutions. Analogously the score equations of $Q$ are 
$$
s \sum_{i=1}^m \theta_Q^{q_i} = 1, \qquad
s \cdot Q_k \cdot (\theta_Q^{q_1}, \ldots, \theta_Q^{q_m}) = b_{Q,k}, \quad k \in [d_2],
$$
and have $\textup{mldeg}(Q)$ solutions. The score equations of $P \times Q$ are given by
\begin{align*}
s \left(\sum_{i=1}^n \theta_P^{p_i}\right) \left(\sum_{i=1}^m \theta_Q^{q_i}\right) &= 1, \\
s \cdot P_k \cdot (\theta_P^{p_1}, \ldots, \theta_P^{p_n}) \left(\sum_{i=1}^m \theta_Q^{q_i}\right)&= b_{P,k}, \quad k \in [d_1], \\
s \cdot Q_k \cdot (\theta_Q^{q_1}, \ldots, \theta_Q^{q_m}) \left(\sum_{i=1}^n \theta_P^{p_i}\right) &= b_{Q,k}, \quad k \in [d_2].
\end{align*}
After solving for $s$ in the first equation for each system, substituting these values of $s$ results in the score equations of $P$ and $Q$, i.e. the system has $ \textup{mldeg} (P) \cdot \textup{mldeg} (Q)$ solutions.
\end{proof}

To apply the corollary, the base case is the one-dimensional cube $C_1=[-1,1]$. While $C_d$ exhibits an ML degree drop for $d \ge 2$, the same is not true for $d=1$, even after considering possible dilates.

\begin{proposition}\label{prop:basecase}
For $t\ge 1$, $tC_1$ exhibits no ML degree drop. In fact, $\textup{deg}(t C_1) = \textup{mldeg}(t C_1) = 2t$.
\end{proposition}

\begin{proof}
The $t$-th dilate of $C_1$ is $[-t,t]$. After a translation, the design matrix is given by
\begin{equation*}
A = \begin{bmatrix}
0 & 1 & 2 & \cdots & 2t
\end{bmatrix}.
\end{equation*}
In dimension one, $\textup{nvol}(tC_1) = \textup{vol}(tC_1) = 2t$. Furthermore, $A$ defines the polynomial
\begin{equation*}
f = 1 + \theta_1 + \theta_1^2 + \ldots + \theta_1^{2t}.
\end{equation*}
By the fundamental theorem of algebra, $f$ has $2t$ complex roots (counted with multiplicity). Let
\begin{equation*}
\xi_{2t+1}(l) = \exp \left(\frac{2 \pi \mathrm{i} l}{2t+1}\right) \quad \textup{for } l \in [2t].
\end{equation*}
The $\xi_{2t+1}(l)$ are pairwise different and $1 + \xi_{2t+1} + \xi_{2t+1}^2 + \ldots + \xi_{2t+1}^{2t} = 0$ for all $l \in [2t]$. Thus $f$ has $2t$ distinct roots. It is well known that a polynomial $p$ of degree $m$ has $m$ distinct roots if and only if $p$ and its derivative have no roots in common. Therefore, $f$ has no singularities and $\textup{mldeg}(t C_1) = 2t$.
\end{proof}

\begin{proof}[Proof of Theorem \ref{thm:mldegcube}]
Combining Corollary \ref{cor:product} and Proposition \ref{prop:basecase} we have
\begin{equation*}
\textup{mldeg}(C_d) = \textup{mldeg}(C_1^d) =  \textup{mldeg}(C_1)^d = 2^d. \qedhere
\end{equation*}
\end{proof}

While this settles the formula for the ML degree of the $d$-cube, we get some insight by analyzing explicitly the score equations and their solutions.

Let $A_d$ be the design matrix defined by $C_d$. For $d>1$, the recursive construction 
\begin{equation} \label{gl3}
A_1 = \begin{bmatrix}
-1 & 0 & 1 \\
\end{bmatrix}, \quad
A_d = \begin{bmatrix}
-1 & 0 & 1 \\
A_{d-1} & A_{d-1} & A_{d-1} \\
\end{bmatrix} \in \mathbb{Z}^{d \times 3^d}
\end{equation}
holds. We are looking at the solutions of 
\begin{equation} \label{gl4}
A'_d \psi(s, \theta) = \frac{1}{u_+} A'_d u
\end{equation}
for generic data $u$. For simplicity, we set
\begin{equation*}
b \coloneqq \frac{1}{u_+} A_d u.
\end{equation*}

Then, by induction on the construction of $A_d$, we have the following score equations.
\begin{lemma}\label{lemma15}
The first equation of \textup{(\ref{gl4})} is of the form
\begin{equation*}
s \prod_{i \in [d]} (1+\theta_i+\theta_{i}^{-1}) = 1.
\end{equation*}
For $k \in [d]$, the $(k+1)$-th equation of \textup{(\ref{gl4})} is of the form
\begin{equation*}
s (\theta_k-\theta_k^{-1}) \prod_{i \in [d]\setminus \{k\}} (1+\theta_i+\theta_i^{-1}) = b_k.
\end{equation*}
\end{lemma}

\begin{proposition}\label{prop:critcube}
With the notation above, the $2^d$ critical points for generic data corresponding to $\textup{mldeg}(\mathrm{C}_d)=2^d$ are given for each $k \in [d]$ by the two possible values
\begin{equation} \label{gl5}
\theta_k = \frac{\pm \sqrt{4-3b_k^2}-b_k}{2(b_k-1)}.
\end{equation}
\end{proposition}

\begin{proof}
By Lemma \ref{lemma15}, 
\begin{equation*}
s = \frac{1}{\prod_{i \in [d]} (1+\theta_i+\theta_{i}^{-1})},
\end{equation*} 
i.e., $s$ is uniquely determined by $\theta_1, \ldots, \theta_d$. Substituting $s$ in equation $k+1$ for all $k \in [d]$ gives the system
\begin{equation*}
\frac{1}{1+\theta_k+\theta_k^{-1}} (\theta_k-\theta_k^{-1}) = b_k, \quad k \in [d]. 
\end{equation*}
Multiplying by $\theta_k$ gives the following quadratic equation whose roots give the desired values:
\begin{equation*}
    (b_k-1)\theta_k^2 + b_k \theta_k + (b_k+1) = 0. \qedhere 
\end{equation*}
\end{proof}

According to \citep[Proposition 3.3]{FamiliesOfPolytopesWithRationalLinearPrecisionInHigherDimensions}, a suitable scaling can be found such that the model associated to $C_3$ has ML degree one. We extend this result by introducing a family of ML degree one scalings for $C_d$, $d\geq 3$. 

Let $c_{k,-1}, c_{k,0}, c_{k,1} \in \mathbb{C}^*$ for $k \in [d]$. Consider the scaling $c \in (\mathbb{C}^*)^n$ given by the product form
\begin{equation} \label{gl6}
c_i = \prod_{k \in [d]} c_{k, a_{ki}},
\end{equation}
where $a_{ki}$ is the $(k,i)$-entry of $A_d \in \mathbb{Z}^{d \times n}$. The general version of Lemma \ref{lemma15} for a scaled model is

\begin{lemma} \label{lemma17}
For $c \in (\mathbb{C}^*)^n$ as in \textup{(\ref{gl6})}, the first equation of $A'_d \psi^c(s,\theta) = [1,b]^T$ is of the form $$s \prod_{i \in [d]} ( c_{i,0} + c_{i,1} \theta_i + c_{i,-1} \theta_{i}^{-1}) = 1.$$ For $k \in [d]$, the $(k+1)$-th equation of $A'_d \psi^c(s,\theta) = [1,b]^T$ is of the form $$s (c_{k,1} \theta_k- c_{k,-1}\theta_k^{-1}) \prod_{i \in [d]\setminus \{k\}} (c_{i,0} + c_{i,1} \theta_i + c_{i,-1} \theta_{i}^{-1}) = b_k.$$

\end{lemma}

\begin{theorem} \label{theorem18}
A scaled model of $C_d$ has ML degree one if $c_{k,-1} = \frac{c_{k,0}^2}{4 c_{k,1}}$ for all $k \in [d]$.
\end{theorem}

\begin{proof}
From Lemma \ref{lemma17}, we have the equations 
\begin{equation} \label{gl7}
\frac{1}{c_{k,0} + c_{k,1} \theta_k + c_{k,-1} \theta_k^{-1}} (c_{k,1} \theta_k - c_{k,-1} \theta_k^{-1}) = b_k
\end{equation}
Equation (\ref{gl7}) has solutions
\begin{equation*}
\theta_k = \frac{\pm \sqrt{(c_{k,0}^2 - 4 c_{k,1} c_{k,-1})b_k^2 + 4 c_{k,1} c_{k,-1}} -c_{k,0} b_k}{2 c_{k,1} (b_k - 1)}.
\end{equation*}
The discriminant is given by $\Delta_{b_k} = 16 c_{k,1} c_{k,-1} ( 4 c_{k,1} c_{k,-1} - c_{k,0}^2)$. It holds that $\Delta_{b_k} = 0$ if and only if $c_{k,-1} = \frac{c_{k,0}^2}{4 c_{k,1}}$ and  $c_{k,1} \neq 0 $. This gives a family of ML degree one scalings since each of these equations has generically one solution
\begin{equation}\label{eq:rational} 
\theta_k = \frac{(b_k+1)c_{k,0}}{2 c_{k,1} (1-b_k)}. \qedhere
\end{equation}
\end{proof}

We illustrate the choice of an ML degree one scaling for the 3-cube using Theorem \ref{theorem18}.

\begin{example}[3-cube] \upshape \label{example20}
The design matrix of $C_3$ is $A_3 \in \mathbb{Z}^{4 \times 27}$ defined in (\ref{gl3}).The components of $\psi^c(s, \theta)$ are given by 
\begin{align*}
p_1 &= c_1 s \theta_1^{-1} \theta_2^{-1} \theta_3^{-1}, \quad &p_{10} &= c_{10} s \theta_2^{-1} \theta_3^{-1}, \quad &p_{19} &= c_{19} s \theta_1 \theta_2^{-1} \theta_3^{-1}, \\
p_2 &= c_2 s \theta_1^{-1} \theta_2^{-1}, \quad &p_{11} &= c_{11} s \theta_2^{-1}, \quad &p_{20} &= c_{20} s \theta_1 \theta_2^{-1}, \\
p_3 &= c_3 s \theta_1^{-1} \theta_2^{-1} \theta_3, \quad &p_{12} &= c_{12} s \theta_2^{-1} \theta_3, \quad &p_{21} &= c_{21} s \theta_1 \theta_2^{-1} \theta_3, \\
p_4 &= c_4 s \theta_1^{-1} \theta_3^{-1}, \quad &p_{13} &= c_{13} s \theta_3^{-1}, \quad &p_{22} &= c_{22} s \theta_1 \theta_3^{-1}, \\
p_5 &= c_5 s \theta_1^{-1}, \quad &p_{14} &= c_{14} s, \quad &p_{23} &= c_{23} s \theta_1, \\
p_6 &= c_6 s \theta_1^{-1} \theta_3, \quad &p_{15} &= c_{15} s \theta_3, \quad &p_{24} &= c_{24} s \theta_1 \theta_3, \\
p_7 &= c_7 s \theta_1^{-1} \theta_2 \theta_3^{-1}, \quad &p_{16} &= c_{16} s \theta_2 \theta_3^{-1}, \quad &p_{25} &= c_{25} s \theta_1 \theta_2 \theta_3^{-1}, \\
p_8 &= c_8 s \theta_1^{-1} \theta_2, \quad &p_{17} &= c_{17} s \theta_2, \quad &p_{26} &= c_{26} s \theta_1 \theta_2, \\
p_9 &= c_9 s \theta_1^{-1} \theta_2 \theta_3, \quad &p_{18} &= c_{18} s \theta_2 \theta_3, \quad &p_{27} &= c_{27} s \theta_1 \theta_2 \theta_3.
\end{align*}
According to Theorem \ref{theorem18} we choose $c_{1,-1} = c_{2,-1} = c_{3,-1} = 1$, $c_{1,0} = c_{2,0} = c_{3,0} = 2$ and $c_{1,1} = c_{2,1} = c_{3,1} = 1$. Using (\ref{gl6}), we get the ML degree one scaling
\begin{alignat*}{7}
c_1 &= c_{1,-1} c_{2,-1} c_{3,-1} &&= 1, \qquad c_{10}  &&= c_{1,0} c_{2,-1} c_{3,-1} &&= 2, \qquad c_{19} &&= c_{1,1} c_{2,-1} c_{3,-1} &&= 1, \\
c_2 &= c_{1,-1} c_{2,-1} c_{3,0} &&= 2, \qquad c_{11} &&= c_{1,0} c_{2,-1} c_{3,0} &&= 4, \qquad c_{20} &&= c_{1,1} c_{2,-1} c_{3,0} &&= 2, \\
c_3 &= c_{1,-1} c_{2,-1} c_{3,1} &&= 1, \qquad c_{12} &&= c_{1,0} c_{2,-1} c_{3,1} &&= 2, \qquad c_{21} &&= c_{1,1} c_{2,-1} c_{3,1} &&= 1, \\
c_4 &= c_{1,-1} c_{2,0} c_{3,-1} &&= 2, \qquad c_{13} &&= c_{1,0} c_{2,0} c_{3,-1} &&= 4, \qquad c_{22} &&= c_{1,1} c_{2,0} c_{3,-1} &&= 2, \\
c_5 &= c_{1,-1} c_{2,0} c_{3,0} &&= 4, \qquad c_{14} &&= c_{1,0} c_{2,0} c_{3,0} &&= 8, \qquad c_{23} &&= c_{1,1} c_{2,0} c_{3,0} &&= 4, \\
c_6 &= c_{1,-1} c_{2,0} c_{3,1} &&= 2, \qquad c_{15} &&= c_{1,0} c_{2,0} c_{3,1} &&= 4, \qquad c_{24} &&= c_{1,1} c_{2,0} c_{3,1} &&= 2, \\
c_7 &= c_{1,-1} c_{2,1} c_{3,-1} &&= 1, \qquad c_{16} &&= c_{1,0} c_{2,1} c_{3,-1} &&= 2, \qquad c_{25} &&= c_{1,1} c_{2,1} c_{3,-1} &&= 1, \\
c_8 &= c_{1,-1} c_{2,1} c_{3,0} &&= 2, \qquad c_{17} &&= c_{1,0} c_{2,1} c_{3,0} &&= 4, \qquad c_{26} &&= c_{1,1} c_{2,1} c_{3,0} &&= 2, \\
c_9 &= c_{1,-1} c_{2,1} c_{3,1} &&= 1, \qquad c_{18} &&= c_{1,0} c_{2,1} c_{3,1} &&= 2, \qquad c_{27} &&= c_{1,1} c_{2,1} c_{3,1} &&= 1.
\end{alignat*}
\end{example}

\begin{remark}
    The statistical interpretation of the model given by $C_d$ with the chosen scaling $c_{k,-1}=1, c_{k,0}=2, c_{k,1}=1$ for $k \in [d]$ is as follows. It consists of $X_1,\dots,X_d$ independent discrete random variables on the state space $\{-1,0,1\}$ with parameter $\theta_k \in (0,1)$ and probability mass function
$$(\mathbb{P}(X_k=-1),\mathbb{P}(X_k=0),\mathbb{P}(X_k=1))=\frac{1}{(1+\theta_k)^2}(1,2\theta_k,\theta_k^2).$$ Then, given data counts $u \in \NN^{3^d}$ and sufficient statistics $b=\frac{1}{u_+} A_d u \in \RR^d$, this model has a rational MLE for each $\theta_k$ given by \eqref{eq:rational}.

\end{remark}
Now we focus on the dual of $C_d$, which is the $d$-dimensional cross-polytope
\begin{equation*}
\cross_d \coloneqq \textup{conv} (\pm e_1, \pm e_2, \ldots, \pm e_d).
\end{equation*}
For $d=3$ it is known as the octahedron shown in Figure \ref{figure4}. In two dimensions it is the polygon of type 4a shown in Figure \ref{figure1} which exhibits no ML degree drop \citep{MaximumLikelihoodEstimationOfToricFanoVarieties}. We generalize this observation to arbitrary dimension. We note that unfortunately Corollary \ref{cor:product} no longer applies. First, a closed formula for the degree of the toric variety is given by the following proposition.

\begin{proposition}
The degree of $\cross_d$ is $2^d$.
\end{proposition}

\begin{proof}
The volume of $\cross_d$ is $2^d/d!$, see e.g. \citep{OnTheVolumeOfProjectionsOfTheCrossPolytope}.
\end{proof}

\begin{figure}
\centering
\begin{subfigure}[b]{0.475\textwidth}
(a)
\end{subfigure}
\hfill
\begin{subfigure}[b]{0.475\textwidth}
(b)
\end{subfigure}
\centering
\begin{subfigure}[b]{0.4\textwidth}
\centering
\hspace{1cm} \includegraphics[scale=0.5]{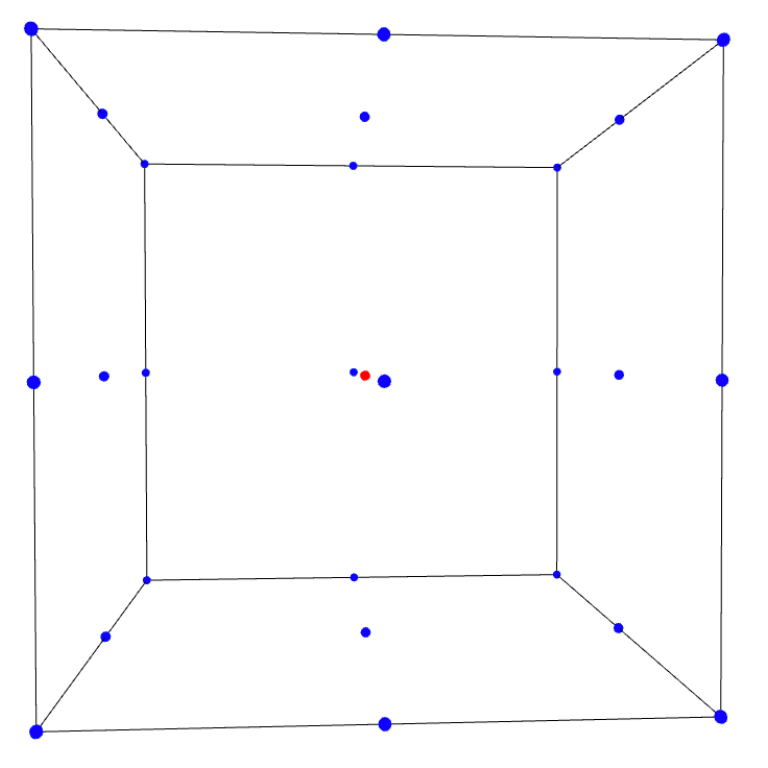}
\end{subfigure}
\hspace{1cm}
\begin{subfigure}[b]{0.4\textwidth}
\centering
\hspace*{0.5cm}\includegraphics[width=\textwidth]{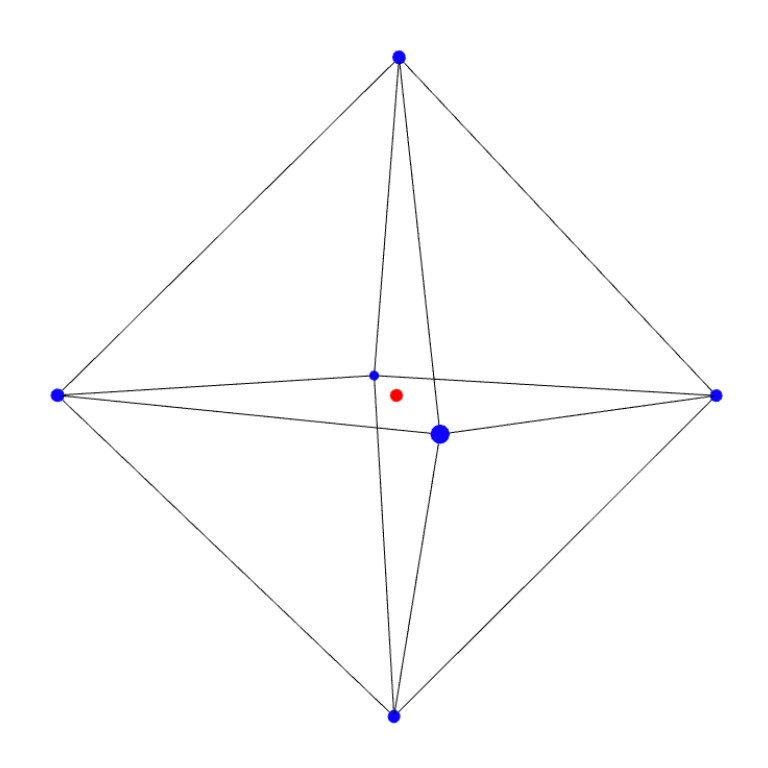}
\end{subfigure}
\caption[]{\small The 3-cube $C_3$ shown in (a) and its dual octahedron $\cross_3$ shown in (b). The red point corresponds to the interior lattice point. The visualizations were created using Polymake \citep{ComputingConvexHullsAndCountingIntegerPointsWithPolymake, polymakeAFrameworkForAnalyzingConvexPolytopes}.} \label{figure4}
\end{figure}

Since $\cross_d$ is reflexive, the design matrix defined by $\cross_d$ is 
\begin{equation*}
A_d = \begin{bmatrix}
0 & e_1 & \ldots & e_d & -e_1 & \ldots & -e_d \\
\end{bmatrix}.
\end{equation*}
With the following theorem we get a generalization of the 2-dimensional result.

\begin{theorem} \label{theorem24}
The ML degree of the cross polytope $\cross_d$ is $2^d$.
\end{theorem}

\begin{proof}
We show that no face of $\cross_d$ contributes a singularity. For $1 \le j \le d$ the $(j-1)$-faces of $\cross_d$ are given by the $j$-subsets of $\{ \pm e_1, \pm e_2, \ldots, \pm e_d \}$ that do not simultaneously contain $e_i$ and $-e_i$. Let $\Gamma = \Gamma^+ \cup \Gamma^-$ be such a subset, where $\Gamma^+$ is the set of non-negative vectors and $\Gamma^-$ is the set of non-positive vectors. Let $I(\Gamma^+)$ and $I(\Gamma^-)$ be the corresponding index sets. The face $\Gamma$ defines the polynomial 
\begin{equation*}
f_\Gamma = \sum_{i \in I(\Gamma^+)} \theta_i + \sum_{i \in I(\Gamma^-)} \frac{1}{\theta_{i}}.
\end{equation*}
To determine the singularities of $f_\Gamma$, we consider the partial derivatives
$$
\frac{\partial f_\Gamma}{\partial \theta_i} (\theta) = 1 \quad \textup{for } i \in I(\Gamma^+), \qquad
\frac{\partial f_\Gamma}{\partial \theta_i} (\theta) = - \frac{1}{\theta_i^2} \quad \textup{for } i \in I(\Gamma^-).
$$
The system 
\begin{equation*}
f_\Gamma = \frac{\partial f_\Gamma}{\partial \theta_i} (\theta) = 0, \quad i \in [d],
\end{equation*}
has no solution $\theta \in (\mathbb{C}^*)^d$. Hence, no $(j-1)$-dimensional face for $1 \le j \le d$ contributes a singularity.

It remains to consider the whole polytope. The corresponding polynomial is given by
\begin{equation*}
f = 1 + \sum_{i \in [d]} \left ( \theta_i + \frac{1}{\theta_i} \right ).
\end{equation*}
Thus,
\begin{equation*}
\frac{\partial f}{\partial \theta_i} (\theta) = 1 - \frac{1}{\theta_i^2} = 0 \quad \textup{if and only if} \quad \theta_i = \pm 1 \quad \textup{for all } i \in [d].
\end{equation*}

We argue that $f \neq 0$ for each $\pm 1$-combination of $\theta_1, \ldots, \theta_d$. For $\theta_i = \pm 1$, $f$ is equivalent to $1 + 2 \sum_{i \in [d]} \theta_i$. Therefore, $f = 0$ if and only if
\begin{equation*}
\sum_{i \in [d]} \theta_i = - \frac{1}{2}.
\end{equation*}
This is a contradiction to $\theta_i = \pm 1$. Thus $\cross_d$ does not imply a singularity. Altogether, $\textup{mldeg} (\cross_d) = \textup{deg} (\cross_d) = 2^d$.
\end{proof}

While the toric variety of $\cross_d$ exhibits no ML degree drop, this is not true for all its scaled toric varieties, as the following example shows.

\begin{example}[ML degree drop for scaled $\cross_2$] \upshape
After an affine transformation, the design matrix of $\cross_2$ is given by
\begin{equation*}
A = \begin{bmatrix}
1 & 2 & 1 & 0 & 1 \\
1 & 1 & 2 & 1 & 0 \\
\end{bmatrix}.
\end{equation*}
Vertices $a_i$ have $\Delta_{a_i} \neq 0$. Each edge $e_i$ has lattice length one and therefore $\Delta_{e_i} \neq 0$. Considering the whole polytope yields
\begin{equation*}
f_c = c_{11} \theta_1 \theta_2 + c_{21} \theta_1^2 \theta_2 + c_{12} \theta_1 \theta_2^2 + c_{01} \theta_2 + c_{10} \theta_1.
\end{equation*}
The $A$-discriminant $\Delta_A$ can be computed using the following \texttt{Macaulay2} code.

\lstset{language=C++, basicstyle=\ttfamily, escapeinside={(*@}{@*)}, mathescape=true}
\begin{lstlisting}
R = QQ[(*@c\_11@*),(*@c\_21@*),(*@c\_12@*),(*@c\_01@*),(*@c\_10@*),(*@t\_1@*),(*@t\_2@*)]
f = (*@c\_11@*)*(*@t\_1@*)*(*@t\_2@*)+(*@c\_21@*)*(*@t\_1@*)^2*(*@t\_2@*)+(*@c\_12@*)*(*@t\_1@*)*(*@t\_2@*)^2+(*@c\_01@*)*(*@t\_2@*)+(*@c\_10@*)*(*@t\_1@*)
I = ideal(f,diff((*@t\_1@*),f),diff((*@t\_2@*),f))
eliminate({(*@t\_1@*),(*@t\_2@*)},I)
\end{lstlisting}
The result is the ideal generated by
\begin{equation*}
\Delta_A = c_{11}^4 c_{01} c_{10} - 8 c_{11}^2 c_{21} c_{01}^2 c_{10} + 16 c_{21}^2 c_{01}^3 c_{10} - 8 c_{11}^2 c_{12} c_{01} c_{10}^2 - 32 c_{21} c_{12} c_{01}^2 c_{10}^2 +16 c_{12}^2 c_{01} c_{10}^3.
\end{equation*}
We have $E_A(c) = 0$ if and only if $\Delta_A = 0$. For $c \in (\mathbb{C}^*)^n$, $\Delta_A = 0$ if and only if
\begin{equation*}
c_{10} = \frac{c_{11}^2 \pm 4 c_{11} \sqrt{c_{21}} \sqrt{c_{01}} + 4 c_{21} c_{01}}{4 c_{12}}.
\end{equation*}
If we choose $c_{11} = 2$, $c_{21} = 4$, $c_{12} = 25$ and $c_{01} = 4$, we get $c_{10} = 1$. Using this scaling, $A' \psi^c (s,\theta) = \frac{1}{u_+} A' u$ has three solutions for generic $u$ according to our computations with \texttt{Julia}. That is, the scaled model has ML degree three.
\end{example}

\section{Reflexive Simplices} \label{section6}

In this section we discuss four classes of reflexive simplices in arbitrary dimension. We give upper bounds on the ML degree by computing degrees of the associated toric varieties, and we conjecture that they are equalities. 

The first family of self-dual reflexive simplices that we consider was introduced by \citet{SelfDualReflexiveSimplicesWithEulerianPolynomials}. For $d \ge 1$, let $\mathcal{Q}_d$ denote the $d$-dimensional simplex
\begin{equation*}
\mathcal{Q}_d \coloneqq \textup{conv} \begin{blockarray}{cccccc}
q_0 & q_1 & q_2 & q_3 & \ldots & q_d\\
\begin{block}{[cccccc]}
1 & -d & 0 & 0 & \ldots & 0 \\
1 & 1 & 1-d & 0 & \ldots & 0 \\
1 & 1 & 1 & 2-d & \ldots & 0 \\
\vdots & \vdots & \vdots & & \ddots & \vdots \\
1 & 1 & 1 & 1 & \ldots & -1 \\
\end{block}
\end{blockarray}.
\end{equation*}

\begin{proposition}\label{prop:refl1}
The degree of $\mathcal{Q}_d$ is $(d+1)!$. In particular, $\mathrm{mldeg}(\mathcal{Q}_d) \leq (d+1)!$.
\end{proposition}

\begin{proof}
The normalized volume of $\mathcal{Q}_d$ is the absolute value of 
\begin{equation*}
\textup{det} \begin{bmatrix}
q_1-q_0 & q_2-q_0 & q_3-q_0 & \ldots & q_d-q_0
\end{bmatrix}. \qedhere
\end{equation*}
\end{proof}

\noindent According to our computations using \texttt{HomotopyContinuation.jl}, the first instances are
$$
\textup{mldeg} (\mathcal{Q}_2) = \textup{deg} (\mathcal{Q}_2) = 6, \qquad
\textup{mldeg} (\mathcal{Q}_3) = \textup{deg} (\mathcal{Q}_3) = 24, \qquad
\textup{mldeg} (\mathcal{Q}_4) = \textup{deg} (\mathcal{Q}_4) = 120.
$$

As can be seen in Figure \ref{figure1}, there are five isomorphism classes of reflexive simplices in two dimensions. A classification of the five isomorphism classes can be studied in \citep[Example 4.7]{VolumeAndLatticePointsOfReflexiveSimplices}. Since $\mathcal{Q}_2$ shown in Figure \ref{figure6} contains seven lattice points, it is isomorphic to $P_{6 \textup{d}}$.

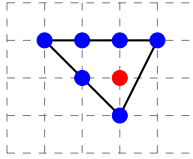
\begin{figure}[h]
\centering
\begin{tikzpicture} 
\draw [step=0.5,gray, dashed] (0,-0.5) grid (2.5,1.5);

\draw[thick] (0.5, 1.0) -- (2.0, 1.0) {};
\draw[thick] (2.0, 1.0) -- (1.5, 0.0) {};
\draw[thick] (1.5, 0.0) -- (0.5, 1.0) {};

\node[draw, circle, inner sep=2pt, fill, blue] at (1.5, 0.0) {};
\node[draw, circle, inner sep=2pt, fill, blue] at (2.0, 1.0) {};
\node[draw, circle, inner sep=2pt, fill, blue] at (1.0, 0.5) {};
\node[draw, circle, inner sep=2pt, fill, blue] at (1.0, 1.0) {};
\node[draw, circle, inner sep=2pt, fill, blue] at (0.5, 1.0) {};
\node[draw, circle, inner sep=2pt, fill, blue] at (1.5, 1.0) {};
\node[draw, circle, inner sep=2pt, fill, red] at (1.5, 0.5) {};
\end{tikzpicture}
\caption{Reflexive simplex $\mathcal{Q}_2$.} \label{figure6}
\end{figure}

The remaining three constructions of reflexive simplices are based on Sylvester's sequence \citep{OnAPointInTheTheoryOfVulgarFractions}. Let $t_1 \coloneqq 2$ and $t_{i+1} \coloneqq t_i^2-t_i+1$. The first terms of the sequence are 
\begin{equation*}
t_1 = 2, \quad t_2 = 3, \quad t_3 = 7, \quad t_4 = 43, \quad t_5 = 1807.
\end{equation*}
The following two simplex constructions were first described by \citet*{OnLatticePolytopesHavingInteriorLatticePoints}. We define
\begin{equation*}
\mathcal{R}_d \coloneqq \textup{conv}(0, t_1 e_1, \ldots, t_d e_d).
\end{equation*}
Then $\mathcal{R}_d$ is reflexive with interior lattice point $\textbf{1}$ \citep{TheReflexiveDimensionOfALatticePolytope}. In two dimensions this simplex is also isomorphic to $P_{6 \textup{d}}$.

\begin{proposition}\label{prop:refl2}
The degree of $\mathcal{R}_d$ is $t_d^2-t_d$.  In particular, $\mathrm{mldeg}(\mathcal{R}_d) \leq t_d^2-t_d$.
\end{proposition}

\begin{proof}
Consider
\begin{equation*}
V_d = \begin{bmatrix}t_1 e_1 & \ldots & t_d e_d \end{bmatrix} \in \mathbb{R}^{d \times d}.
\end{equation*}
Since $\textup{nvol}(\mathcal{R}_d) = \vert \textup{det} V_d \vert$ we show $\vert \textup{det} V_d \vert = t_d^2-t_d$ by induction. For $d=2$,
\begin{equation*}
\textup{det} V_2 = t_1 t_2 = 6 = t_2^2-t_2.
\end{equation*}
We assume that the formula holds in dimension $d$. Using the definition of Sylvester's sequence,
\begin{equation*}
\textup{det} V_{d+1} = \prod_{i=1}^{d+1} t_i = \textup{det} (V_d) t_{d+1} = (t_d^2-t_d)t_{d+1} = (t_{d+1}-1)t_{d+1} = t_{d+1}^2-t_{d+1}.\qedhere
\end{equation*} 
\end{proof}

\noindent We are able to verify that
$$
\textup{mldeg} (\mathcal{R}_2) = \textup{deg} (\mathcal{R}_2) = 6, \qquad
\textup{mldeg} (\mathcal{R}_3) = \textup{deg} (\mathcal{R}_3) = 42,
$$
but since the number of lattice points grows very fast, it was not possible to compute further ML degrees of $\mathcal{R}_d$ in higher dimensions.

A small modification of $\mathcal{R}_d$ defines the reflexive simplex 
\begin{equation*}
\mathcal{S}_d \coloneqq \textup{conv}(0, t_1 e_1, \ldots, t_{d-1} e_{d-1}, 2(t_d-1)e_d).
\end{equation*}
In two dimensions it is isomorphic to $P_{8\textup{c}}$.

\begin{proposition}\label{prop:refl3}
For any $d>1$, $\mathrm{mldeg}(\mathcal{S}_d) \leq (t_{d-1}^2-t_{d-1})2(t_d-1) = \mathrm{deg}(\mathcal{S}_d)$.
\end{proposition}

\begin{proof}
The normalized volume of $\mathcal{S}_d$ is $\textup{det}[t_1 e_1, \ldots, t_{d-1} e_{d-1}, (2t_d-2)e_d]$.
\end{proof}

\noindent The first two instances are:
$$
\textup{mldeg} (\mathcal{S}_2) = \textup{deg} (\mathcal{S}_2) = 8, \qquad
\textup{mldeg} (\mathcal{S}_3) = \textup{deg} (\mathcal{S}_3) = 72.
$$

The fourth family of reflexive simplices was introduced in \citep[Theorem 3.2]{TheDVectorsOfReflexivePolytopesAndOfTheDualPolytopes}. For $d \ge 3$ we consider the $d$-dimensional simplex $\mathcal{T}_d$ whose vertices $v_i \in \mathbb{R}^d$ are of the form 
\begin{equation*}
v_i = \begin{cases}
-3 e_1 - 2 \sum_{i=2}^d e_i, & i = 0, \\
e_1, & i=1, \\
e_1 + 2 e_i, & i = 2,3, \\
e_1 + 2 t_{i-3} e_i, & i = 4, \ldots, d.
\end{cases}
\end{equation*}

\begin{proposition}\label{prop:refl4}
For $d\geq 3$, 
\begin{equation*}
\mathrm{mldeg}(\mathcal{T}_d) \leq 
2^{d+1} \prod_{i=1}^{d-3} t_i = \textup{deg} (\mathcal{T}_d).
\end{equation*}
\end{proposition}

\begin{proof}
The normalized volume of $\mathcal{T}_d$ is the absolute value of 
\begin{align*}
\textup{det} \begin{bmatrix}
-4 & 0 & 0 & 0 & 0 & \ldots & 0 \\
-2 & 2 & 0 & 0 & 0 & \ldots & 0 \\
-2 & 0 & 2 & 0 & 0 & \ldots & 0 \\
-2 & 0 & 0 & 2t_1 & 0 & \ldots & 0 \\
-2 & 0 & 0 & 0 & 2t_2 & \ldots & 0 \\
\vdots & \vdots & \vdots & \vdots & & \ddots & \vdots \\
-2 & 0 & 0 & 0 & 0 & \ldots & 2t_{d-3} \\
\end{bmatrix}. \alignqedhere
\end{align*}
\end{proof}

\noindent In this case, the first two instances are
$$
\textup{mldeg} (\mathcal{T}_3) = \textup{deg} (\mathcal{T}_3) = 16, \qquad
\textup{mldeg} (\mathcal{T}_4) = \textup{deg} (\mathcal{T}_4) = 64.
$$

Based on our computations, we state the following conjecture.

\begin{conjecture}
The reflexive simplices $\mathcal{Q}_d, \mathcal{R}_d, \mathcal{S}_d, \mathcal{T}_d$ do not exhibit an ML degree drop with the standard scaling. In particular, all inequalities in Propositions \ref{prop:refl1}, \ref{prop:refl2}, \ref{prop:refl3} and \ref{prop:refl4} are in fact equalities. 
\end{conjecture}
Notwithstanding the conjecture, we point out that it is possible that a reflexive simplex exhibits an ML degree drop with the standard scaling, as the following example shows.


\begin{example} \upshape
Consider the four-dimensional reflexive simplex with vertices $(1,0,0,0)$, $(0,1,0,0)$, $(0,1,3,0)$, $(2,5,0,6)$ and $(-4,-7,-3,-6)$. The degree of the corresponding toric variety is 108, while the ML degree is 107. This can be seen by determining the singularities of the $f$-polynomials defined by all faces. We first consider the facet with design matrix $$A = \begin{bmatrix} 6 & 5 & 5 & 4 & 4 & 4 & 4 & 4 & 4 & 4 & 3 & 3 & 2 & 2 & 2 & 0 \\ 3 & 4 & 3 & 6 & 5 & 4 & 4 & 3 & 2 & 3 & 3 & 2 & 3 & 2 & 1 & 0 \\ 12 & 9 & 9 & 6 & 6 & 7 & 6 & 7 & 8 & 6 & 5 & 5 & 3 & 3 & 4 & 0\end{bmatrix}.$$ The corresponding polynomial $f$ has one singularity as the following computation with \texttt{Macaulay2} \cite{Macaulay2} verifies; the last command returns the Gröbner basis $\{t_1-1, t_2-1, t_3+1\}$.
\begin{verbatim}
    R = QQ[t_1,t_2,t_3]
    f = t_1^6*t_2^3*t_3^12 + t_1^5*t_2^4*t_3^9 + t_1^5*t_2^3*t_3^9 + 
         t_1^4*t_2^6*t_3^6 + t_1^4*t_2^5*t_3^6 + t_1^4*t_2^4*t_3^7 + 
         t_1^4*t_2^4*t_3^6 + t_1^4*t_2^3*t_3^7 + t_1^4*t_2^2*t_3^8 + 
         t_1^4*t_2^3*t_3^6 + t_1^3*t_2^3*t_3^5 + t_1^3*t_2^2*t_3^5 + 
         t_1^2*t_2^3*t_3^3 + t_1^2*t_2^2*t_3^3 + t_1^2*t_2*t_3^4 + 1
    I = ideal(f,diff(t_1,f),diff(t_2,f),diff(t_3,f))
    gens gb I
\end{verbatim}
For all other faces the Gröbner basis is $\{1\}$, so that there is only one singularity in total. \hfill \mbox{\Large$\blacktriangleleft$}
\end{example}

\section{Constructing Reflexive Polytopes} \label{section7}

In Section~\ref{section5} we studied the ML degrees of hypercubes and cross polytopes in any dimension. In this section, we present generalizations of these two families and examine their ML degree.

Given a reflexive polytope $P$ of dimension $d \ge 1$, define
\begin{align*}
\mathcal{A}(P) &\coloneqq P \times [-1,1], \\
\mathcal{B}(P) &\coloneqq \textup{conv} (P \times \{0\}, (0, 0, \ldots, 0, 1), (0, 0, \ldots, 0, -1)),\\
\mathcal{C}(P) &\coloneqq \textup{conv} (P \times [-1,0], (0, 0, \ldots, 0, 1)).
\end{align*}
While $\mathcal{A}$ corresponds to the cube construction, $\mathcal{B}$ is the construction of the cross polytope (using $P = C_d$ and $P = \cross_d$, respectively). Both constructions yield a reflexive polytope of dimension $d+1$. Additionally, \citet{TheDVectorsOfReflexivePolytopesAndOfTheDualPolytopes} showed that $\mathcal{C}(P)$ is reflexive. Since the design matrix is determined by all lattice points, the following lemma is helpful to specify $A$.

\begin{lemma}
Let $P$ be a lattice polytope containing $n$ lattice points. Then $\mathcal{A}(P)$ contains $3n$ lattice points, $\mathcal{B}(P)$ contains $n+2$ lattice points and $\mathcal{C}(P)$ contains $2n+1$ lattice points.
\end{lemma}

\begin{proof}
Follows from the respective constructions.
\end{proof}

We introduce the following notation for the repeated application of $\mathcal{A}$, $\mathcal{B}$ and $\mathcal{C}$. For $k \in \NN$ and a reflexive polytope $P$ of dimension $d \ge 1$, we define the $(d+k)$-dimensional polytope
\begin{equation*}
\mathcal{A}^{k} (P) \coloneqq \underbrace{\mathcal{A}(\mathcal{A}(\ldots \mathcal{A}(\mathcal{A}}_{k \textup{ times}}(P))\ldots))
\end{equation*}
by applying construction $\mathcal{A}$ exactly $k$ times. Note $\mathcal{A}^ 0(P) = P$. The polytopes $\mathcal{B}^{k} (P)$ and $\mathcal{C}^{k} (P)$ are defined analogously.

We begin our investigations by specifying an upper bound for the ML degree.

\begin{proposition} \label{proposition26}
Let $P$ be a lattice polytope of dimension $d \ge 1$. For $k \ge 1$,
\begin{align*}
\textup{deg}(\mathcal{A}^{k}(P)) &= 2^k \cdot \frac{(d+k)!}{d!} \cdot \textup{deg}(P), \\
\textup{deg}(\mathcal{B}^{k}(P)) &= 2^k \cdot \textup{deg}(P), \\
\textup{deg}(\mathcal{C}^{k}(P)) &= \frac{(d+k+1)!}{(d+1)!} \cdot \textup{deg}(P).
\end{align*}
\end{proposition}

\begin{proof}
By construction $\mathcal{A}$, $\textup{vol} (\mathcal{A}^k(P)) = 2^k \cdot \textup{vol} (P)$. It follows
\begin{equation*}
\textup{nvol} (\mathcal{A}^k(P)) = 2^k \cdot \frac{(d+k)!}{d!} \cdot \textup{nvol} (P).
\end{equation*}

The statement for construction $\mathcal{B}$ follows analogously since $\mathcal{B}(P)$ is a bipyramid over $P$. Therefore,
\begin{equation*}
\textup{vol}(\mathcal{B}^k(P)) = 2^k \cdot \frac{d!}{(d+k)!} \cdot \textup{vol}(P).
\end{equation*}
It follows $\textup{nvol}(\mathcal{B}^k(P)) = 2^k \cdot \textup{nvol}(P)$.

The volume of $\mathcal{C}(P)$ is the sum of the volume of $P \times [-1,0]$ and a pyramid over P. Therefore,
\begin{equation*}
\textup{vol} (\mathcal{C}^k(P)) = \frac{d+k+1}{d+1} \cdot \textup{vol}(P).
\end{equation*}
It follows
\begin{equation*}
\textup{nvol}  (\mathcal{C}^k(P)) = \frac{(d+k+1)!}{(d+1)!} \cdot \textup{nvol}(P). \qedhere
\end{equation*} 
\end{proof}

If $P$ exhibits an ML degree drop, then Proposition \ref{proposition26} gives a strict upper bound on the ML degree of $\mathcal{A}(P)$ and $\mathcal{C}(P)$.

\begin{proposition}
Let $P$ be a lattice polytope. If $P$ exhibits an ML degree drop, then $\mathcal{A}(P)$ and $\mathcal{C}(P)$ exhibit an ML degree drop. In particular, $\textup{mldrop} (\mathcal{A}(P)), \textup{mldrop} (\mathcal{C}(P)) \ge \textup{mldrop}(P)$.
\end{proposition}

\begin{proof}
Polytope $\mathcal{A}(P)$ contains the facets $\Gamma_1 \coloneqq P \times \{-1\}$ and $\Gamma_2 \coloneqq P \times \{ 1 \}$. Since $P$ exhibits an ML degree drop, $\Delta_{\Gamma_1} = \Delta_{\Gamma_2} = 0$ and therefore $E_{\mathcal{A}(P)}(c) = 0$ for the standard scaling. Polytope $\mathcal{C}(P)$ contains the facet $\Gamma_1 \coloneqq P \times \{-1\}$ and therefore $E_{\mathcal{C}(P)}(c) = 0$ for the standard scaling. In particular, $\mathcal{A}(P)$ and $\mathcal{C}(P)$ imply at least as many singularities as $P$.
\end{proof}

\begin{figure}[H]
\centering
\begin{subfigure}[b]{0.32\textwidth}
\centering
\includegraphics[width=\textwidth]{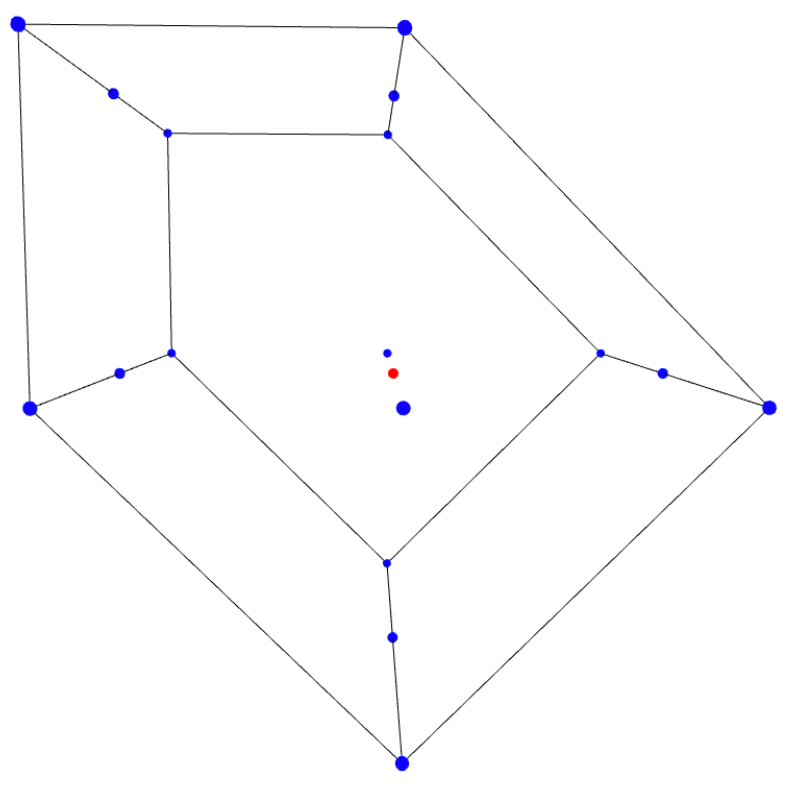}
\end{subfigure}
\hfill
\begin{subfigure}[b]{0.32\textwidth}
\centering
\includegraphics[width=\textwidth]{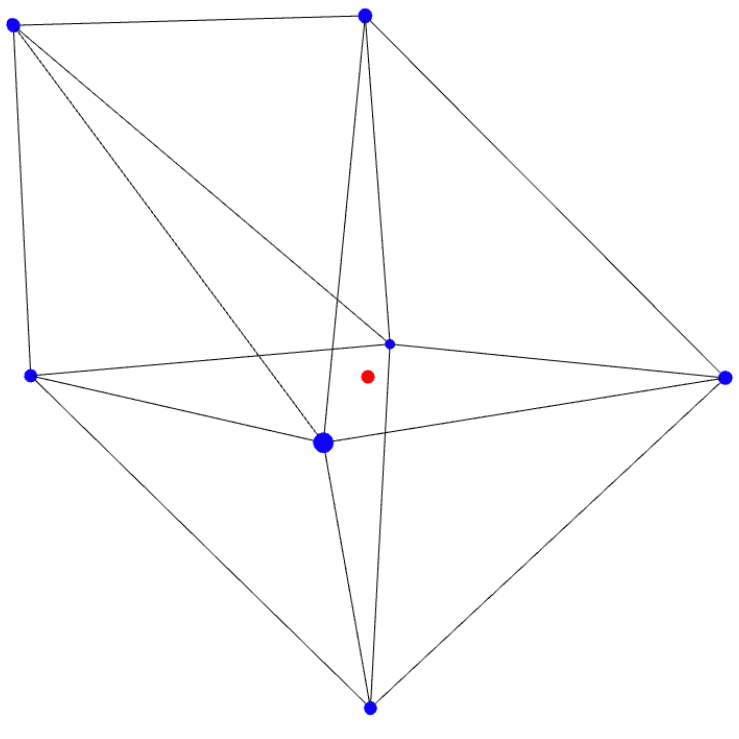}
\end{subfigure}
\hfill
\begin{subfigure}[b]{0.32\textwidth}
\centering
\includegraphics[width=\textwidth]{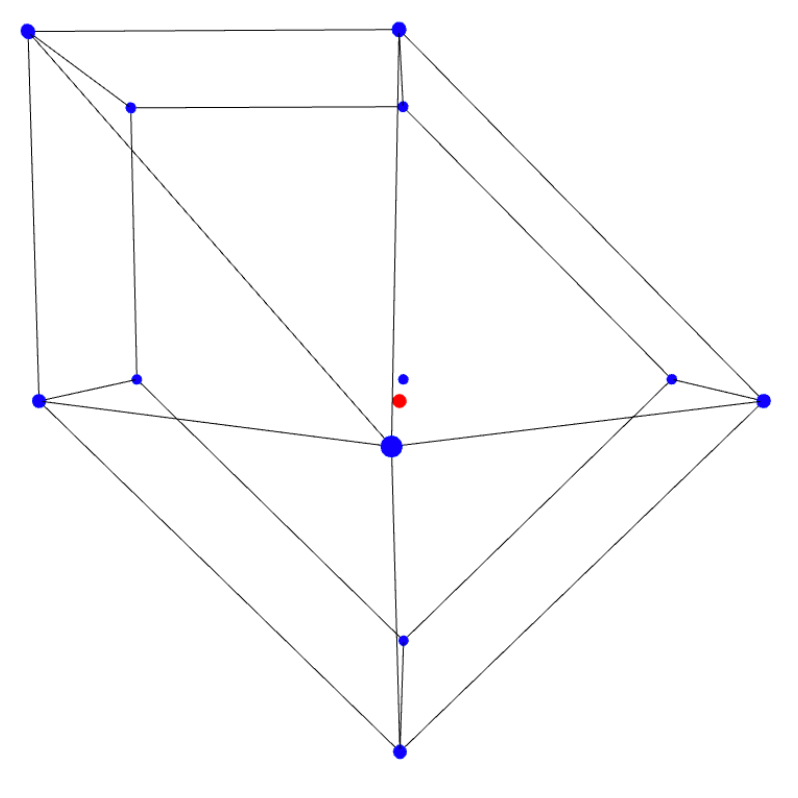}
\end{subfigure}
\centering
\begin{subfigure}[b]{0.32\textwidth}
(a) Polytope $\mathcal{A}(P_{5\textup{a }})$.
\end{subfigure}
\hfill
\begin{subfigure}[b]{0.32\textwidth}
(b) Polytope $\mathcal{B}(P_{5\textup{a }})$.
\end{subfigure}
\hfill
\begin{subfigure}[b]{0.32\textwidth}
(c) Polytope $\mathcal{C}(P_{5\textup{a }})$.
\end{subfigure}
\caption[]{\small Constructions $\mathcal{A}$, $\mathcal{B}$ and $\mathcal{C}$ based on the reflexive polygon of type 5a. The red point corresponds to the interior lattice point. The visualizations were created using Polymake \citep{ComputingConvexHullsAndCountingIntegerPointsWithPolymake, polymakeAFrameworkForAnalyzingConvexPolytopes}.} \label{figure5}
\end{figure}

We computed the degree and the ML degree for the three constructions based on the 16 reflexive polygons. The results are shown in Table \ref{table4}. While the occurrence of an ML degree drop is preserved under construction $\mathcal{A}$ and $\mathcal{C}$, the reflexive polygons of type 5a and 8a show that this is not true~for~$\mathcal{B}$.

\begin{table}[h!]
\begin{minipage}{0.49\textwidth}
\centering
\begin{tabular}{c c c} 
Polytope & mldeg(V) & deg(V) \\ [0.2ex] 
\hline \\ [-1ex] 
$P_3$ & 3 & 3 \\ 
$\mathcal{A}(P_3)$ & 6 & 18 \\
$\mathcal{B}(P_3)$ & 6 & 6 \\
$\mathcal{C}(P_3)$ & 8 & 12 \\ [1ex] 

$P_{4\textup{a}}$ & 4 & 4 \\ 
$\mathcal{A}(P_{4\textup{a}})$ & 8 & 24 \\
$\mathcal{B}(P_{4\textup{a}})$ & 8 & 8 \\
$\mathcal{C}(P_{4\textup{a}})$ & 12 & 16 \\ [1ex] 

$P_{4\textup{b}}$ & 4 & 4 \\ 
$\mathcal{A}(P_{4\textup{b}})$ & 8 & 24 \\
$\mathcal{B}(P_{4\textup{b}})$ & 8 & 8 \\
$\mathcal{C}(P_{4\textup{b}})$ & 12 & 16 \\ [1ex] 

$P_{4\textup{c}}$ & 4 & 4 \\ 
$\mathcal{A}(P_{4\textup{c}})$ & 8 & 24 \\
$\mathcal{B}(P_{4\textup{c}})$ & 8 & 8 \\
$\mathcal{C}(P_{4\textup{c}})$ & 12 & 16 \\ [1ex] 

$P_{5\textup{a}}$ & 3 & 5 \\ 
$\mathcal{A}(P_{5\textup{a}})$ & 6 & 30 \\
$\mathcal{B}(P_{5\textup{a}})$ & 10 & 10 \\
$\mathcal{C}(P_{5\textup{a}})$ & 11 & 20 \\ [1ex] 

$P_{5\textup{b}}$ & 5 & 5 \\ 
$\mathcal{A}(P_{5\textup{b}})$ & 10 & 30 \\
$\mathcal{B}(P_{5\textup{b}})$ & 10 & 10 \\
$\mathcal{C}(P_{5\textup{b}})$ & 15 & 20 \\ [1ex] 

$P_{6\textup{a}}$ & 6 & 6 \\ 
$\mathcal{A}(P_{6\textup{a}})$ & 12 & 36 \\
$\mathcal{B}(P_{6\textup{a}})$ & 10 & 12 \\
$\mathcal{C}(P_{6\textup{a}})$ & 18 & 24 \\ [1ex] 

$P_{6\textup{b}}$ & 6 & 6 \\ 
$\mathcal{A}(P_{6\textup{b}})$ & 12 & 36 \\
$\mathcal{B}(P_{6\textup{b}})$ & 12 & 12 \\
$\mathcal{C}(P_{6\textup{b}})$ & 18 & 24 \\ [1ex] 
\end{tabular}
\end{minipage}
\begin{minipage}{0.49\textwidth}
\centering
\begin{tabular}{c c c} 
Polytope & mldeg(V) & deg(V) \\ [0.2ex] 
\hline \\ [-1ex] 
$P_{6\textup{c}}$ & 6 & 6 \\ 
$\mathcal{A}(P_{6\textup{c}})$ & 12 & 36 \\
$\mathcal{B}(P_{6\textup{c}})$ & 12 & 12 \\
$\mathcal{C}(P_{6\textup{c}})$ & 18 & 24 \\ [1ex] 

$P_{6\textup{d}}$ & 6 & 6 \\ 
$\mathcal{A}(P_{6\textup{d}})$ & 12 & 36 \\
$\mathcal{B}(P_{6\textup{d}})$ & 11 & 12 \\
$\mathcal{C}(P_{6\textup{d}})$ & 18 & 24 \\ [1ex] 

$P_{7\textup{a}}$ & 7 & 7 \\ 
$\mathcal{A}(P_{7\textup{a}})$ & 14 & 42 \\
$\mathcal{B}(P_{7\textup{a}})$ & 14 & 14 \\
$\mathcal{C}(P_{7\textup{a}})$ & 21 & 28 \\ [1ex] 

$P_{7\textup{b}}$ & 7 & 7 \\ 
$\mathcal{A}(P_{7\textup{b}})$ & 14 & 42 \\
$\mathcal{B}(P_{7\textup{b}})$ & 14 & 14 \\
$\mathcal{C}(P_{7\textup{b}})$ & 21 & 28 \\ [1ex] 

$P_{8\textup{a}}$ & 4 & 8 \\ 
$\mathcal{A}(P_{8\textup{a}})$ & 8 & 48 \\
$\mathcal{B}(P_{8\textup{a}})$ & 16 & 16 \\
$\mathcal{C}(P_{8\textup{a}})$ & 16 & 32 \\ [1ex] 

$P_{8\textup{b}}$ & 8 & 8 \\ 
$\mathcal{A}(P_{8\textup{b}})$ & 16 & 48 \\
$\mathcal{B}(P_{8\textup{b}})$ & 16 & 16 \\
$\mathcal{C}(P_{8\textup{b}})$ & 24 & 32 \\ [1ex] 

$P_{8\textup{c}}$ & 8 & 8 \\ 
$\mathcal{A}(P_{8\textup{c}})$ & 16 & 48 \\
$\mathcal{B}(P_{8\textup{c}})$ & 16 & 16 \\
$\mathcal{C}(P_{8\textup{c}})$ & 24 & 32 \\ [1ex] 

$P_9$ & 9 & 9 \\ 
$\mathcal{A}(P_9)$ & 18 & 54 \\
$\mathcal{B}(P_9)$ & 18 & 18 \\
$\mathcal{C}(P_9)$ & 27 & 36 \\ [1ex] 
\end{tabular}
\end{minipage}
\caption{The ML degree and the degree of $\mathcal{A}(P), \mathcal{B}(P)$ and $\mathcal{C}(P)$ for each reflexive polygon $P$.} \label{table4}
\end{table}

\begin{example}[Type 5a polygon] \upshape
According to \cite{MaximumLikelihoodEstimationOfToricFanoVarieties}, $\textup{deg} (P_{5\textup{a}}) = 5$ and $\textup{mldeg} (P_{5\textup{a }}) = 3$. The Polytopes $\mathcal{A}(P_{5\textup{a}})$, $\mathcal{B}(P_{5\textup{a}})$ and $\mathcal{C}(P_{5\textup{a}})$ are shown in Figure \ref{figure5}. The corresponding ML degrees and degrees of the toric variety are listed in Table \ref{table4}. In particular, $\textup{deg} (\mathcal{B}(P_{5\textup{a}})) = \textup{mldeg} (\mathcal{B}(P_{5\textup{a}})) = 10$. \end{example}

The multiplicity of the construction $\mathcal{A}$ shown in Table \ref{table4} is explained by the following corollary. Since $\mathcal{A}$ is related to the cube, we use the ideas and notations from Section \ref{section4}.

\begin{corollary} \label{theorem29}
Let P be a lattice polytope of dimension $d \ge 1$. For $k \in \NN$,
\begin{equation*}
\textup{mldeg}(\mathcal{A}^k(P)) = 2^k \textup{mldeg} (P). 
\end{equation*} 
\end{corollary}
\begin{proof}
We note that since $\mathcal{A}(P) = P \times [-1,1]$, then $\mathcal{A}^k(P) = P \times C_k$, where $C_k$ is a $k$-dimensional cube in the appropriate dimension. Then combining Corollary \ref{cor:product} with Theorem \ref{thm:mldegcube} we have that
\begin{equation*}
   \textup{mldeg}(\mathcal{A}^k(P)) = \textup{mldeg}(P) \times \textup{mldeg}(C_k) = 2^k \textup{mldeg} (P). \qedhere
\end{equation*} 
\end{proof}

Using the ML degree one scaling of $C_d$ explained in Theorem \ref{theorem18}, we get the following statement for construction $\mathcal{A}$.

\begin{theorem}
Let $P$ be a lattice polytope of dimension $d \ge 1$. For $k \ge 1$ there exists a scaling such that the scaled model of $\mathcal{A}^k (P)$ has ML degree $\textup{mldeg} (P)$.
\end{theorem}

\begin{proof}
Let $A$ be the design matrix of $P$. For generic $u$, the system $A' \psi_{A} (s, \theta) = [1,b]^T$ has $\textup{mldeg} (P)$ solutions and is of the form
$$
s \cdot f = 1, \qquad
s \cdot p_i = b_i, \quad i \in [d], 
$$
where $p_1, p_2, \ldots, p_d$ are Laurent polynomials in $\theta_1, \theta_2, \ldots, \theta_d$.

Let $A_k$ be the design matrix of $\mathcal{A}^k (P)$. For $i \in [d+1:d+k]$, let $c_{i,-1}, c_{i,0}, c_{i,1} \in \mathbb{C}^*$ such that $c_{i,-1} = c_{k,0}^2/(4 c_{k,1})$. Consider the scaling $c$ with
\begin{equation*}
c_j = \prod_{i=d+1}^{d+k} c_{i,a_{ij}},
\end{equation*} 
where $a_{ij}$ is the $(i,j)$-entry of $A_k$. Analogously to Lemma \ref{lemma17}, we write $A'_k \psi_{A_k}^c (s, \theta) = [1,b]^T$ as
\begin{align}
f^{-1} p_i &= b_i, \quad i \in [d] \notag \\ 
\frac{1}{c_{i,0} + c_{i,1} \theta_i + c_{i,-1} \theta_i^{-1}} (c_{i,1} \theta_i -  c_{i,-1} \theta_i^{-1}) &= b_i, \quad i \in [d+1:d+k]. \label{gl18} 
\end{align}
Each equation of the form (\ref{gl18}) has exactly one solution (see proof of Theorem \ref{theorem18}). Altogether the system has $\textup{mldeg} (P)$ solutions for generic data.
\end{proof}

We continue with construction $\mathcal{B}$. Applying $\mathcal{B}$ once to each reflexive polygon, only $\mathcal{B}(P_{6\textup{a}})$ and $\mathcal{B}(P_{6\textup{d}})$ exhibit an ML degree drop (see Table \ref{table4}). The polygons $P_{6\textup{a}}$ and $P_{6\textup{d}}$ themselves have no ML degree drop. Since the ML degrees of $\mathcal{B}(P_{6\textup{a}})$ and $\mathcal{B}(P_{6\textup{d}})$ are different while $\textup{mldeg}(P_{6 \textup{a}}) = \textup{mldeg}(P_{6 \textup{b}})$, we do not expect that it is possible to give a universal formula for $\textup{mldeg} (\mathcal{B}^k(\cdot))$. Therefore, studying the ML degree of construction $\mathcal{B}$ requires a case-by-case study. We use the following general observation.

Given a polytope $P$ of dimension $d$ and $k\in \NN$, let $f_{d+k} \in \mathbb{Z} [\theta_1^{\pm 1}, \ldots, \theta_{d+k}^{\pm 1}]$ be the Laurent polynomial defined by the design matrix of $\mathcal{B}^k(P)$.

\begin{lemma} \label{lemma7.7}
Let $P$ be a lattice polytope of dimension $d$ and let $\theta = (\theta_1, \ldots, \theta_d)$ be a common root of the partial derivatives of $f_d$. Then $\mathcal{B}^k(P)$ exhibits an ML degree drop if $(\theta_{d+1}, \ldots, \theta_{d+k}) \in \{\pm 1\}^k$ exists such that
\begin{equation*}
\sum_{i=d+1}^{d+k} \theta_i = \frac{-f_d (\theta)}{2}.
\end{equation*}
\end{lemma}

\begin{proof}
By construction $\mathcal{B}$,
\begin{equation*}
f_{d+k} = f_d + \sum_{i=d+1}^{d+k} \left ( \theta_i + \frac{1}{\theta_i} \right ).
\end{equation*}
Therefore, $\partial_{\theta_i} f_{d+k} = \partial_{\theta_i} f_d$ for all $i \in [d]$. Furthermore,
\begin{equation*}
\frac{\partial f_{d+k}}{\partial \theta_i} (\theta_1, \ldots, \theta_{d+k}) = 1 - \frac{1}{\theta_i^2} = 0 \quad \Leftrightarrow \quad \theta_i = \pm 1 \quad \textup{for all } i \in [d+1:d+k].
\end{equation*}
Now let $\theta = (\theta_1,\ldots, \theta_d)$ be a common root of $\partial_{\theta_i} f_d$, $i \in [d]$. Then all partial derivatives of $f_{d+k}$ have a common root. To check whether this is a root of $f_{d+k}$, we consider 
\begin{equation*}
f_{d+k} = f_d + 2 \sum_{i=d+1}^{d+k} \theta_i
\end{equation*}
for $\theta_i = \pm 1$. Thus, given that $\theta_i = \pm 1$ for all $i \in [d+1:d+k]$,
\begin{equation*}
f_{d+k}(\theta_1, \ldots, \theta_{d+k}) = 0 \quad \Leftrightarrow \quad \sum_{i=d+1}^{d+k} \theta_i = \frac{- f_d (\theta)}{2}. \qedhere
\end{equation*}
\end{proof}

In particular, the following theorem holds for the (multiple) application of construction $\mathcal{B}$ to a reflexive polygon. Let $\# \textup{sing}(f_{d+k})$ denote the number of singularities of $f_{d+k}$.

\begin{theorem}
Let $P$ be a reflexive polygon. For $k \in \NN$,
\begin{equation*}
\textup{mldrop} (\mathcal{B}^{k}(P)) = \# \textup{sing}(f_{2+k}).
\end{equation*}
\end{theorem}

\begin{proof}
It is sufficient to show that no proper face contributes a singularity to an ML degree drop. Let $P$ be a reflexive polygon. By construction, $\mathcal{B}^{k}(P)$ has the edges of $P$ as one-dimensional faces. According to the results of \cite[Example 4.4]{MaximumLikelihoodEstimationOfToricFanoVarieties} and the proof of Proposition \ref{proposition10}, these do not contribute any singularity. All other faces contain either $e_i$ or $-e_i$ as a vertex for at least one $i \in \{ 3, \ldots, 2+k\}$. Let $\Gamma$ be a face containing $e_i$ or $-e_i$. Then
\begin{equation*}
\frac{\partial f_\Gamma}{\partial \theta_i} (\theta) = \begin{cases}
1, & \textup{if } \Gamma \textup{ contains } e_i, \\
- \frac{1}{\theta_i^2}, & \textup{if } \Gamma \textup{ contains } -e_i.
\end{cases}
\end{equation*}
That is, $f_\Gamma$ cannot have a singularity.
\end{proof}

The common roots of $\partial_{\theta_i} f_d$ thus provide important information about the occurrence of an ML degree drop regarding $\mathcal{B}^{k}(P)$. Based on this observation, we state a series of results regarding the 16 reflexive polygons. All common roots of $\partial_{\theta_1} f_2$ and $\partial_{\theta_2} f_2$ were computed using \texttt{Mathematica} and can be found in the \texttt{MathRepo} webpage. The associated ML degrees are given in Table \ref{table5}.

\begin{theorem}\label{cor:allB}
\begin{itemize}
\item[(i)] Let $k \ge 2$ be even. Then
\begin{equation*}
\textup{mldrop} (\mathcal{B}^{k} (P_3)) = \binom{k}{\frac{k-2}{2}}, \qquad \textup{mldrop} (\mathcal{B}^{k} (P_{5 \textup{a}})) = 2 \cdot \binom{k}{\frac{k}{2}}, \qquad \textup{mldrop} (\mathcal{B}^{k} (P_{8 \textup{a}})) = 4 \cdot \binom{k}{\frac{k}{2}}.
\end{equation*}
\item[(ii)] Let $k \ge 1$ (resp. $k \ge 5$) be odd. Then
\begin{equation*}
\textup{mldrop} (\mathcal{B}^{k} (P_{6 \textup{a}})) = 2 \cdot \binom{k}{\frac{k-1}{2}}, \qquad \textup{mldrop} (\mathcal{B}^{k} (P_{6 \textup{d}})) = \binom{k}{\frac{k-1}{2}}, \qquad \textup{mldrop} (\mathcal{B}^{k} (P_9)) = \binom{k}{\frac{k-5}{2}}.
\end{equation*}
\item[(iii)] For $k\geq 1$, $\textup{mldrop}(\mathcal{B}^k(Q))=0$ for $Q \in 
\{P_{4\textup{a}}, P_{4\textup{b}}, P_{4\textup{c}}, P_{5\textup{b}}, P_{6\textup{b}},  P_{6\textup{c}},  P_{7\textup{a}},  P_{7\textup{b}}, P_{8\textup{b}}, P_{8\textup{c}} \}.$
\end{itemize}
\end{theorem}

\begin{proof}
\begin{itemize}
\item[(i)] According to Lemma \ref{lemma7.7}, it is sufficient to consider 
\begin{equation*}
\sum_{i=3}^{2+k} \theta_i = \frac{- f_2(\theta_1,\theta_2)}{2} = -2 \qquad \sum_{i=3}^{2+k} \theta_i = \frac{- f_2(\theta_1,\theta_2)}{2} = 0 \qquad
\sum_{i=3}^{2+k} \theta_i = \frac{- f_2(\theta_1,\theta_2)}{2} = 0
\end{equation*}
with $\theta_i = \pm 1$ for all $i \in [3:2+k]$ for one, two and four common roots, respectively.
\item[(ii)] In this case we consider 
\begin{equation*}
\sum_{i=3}^{2+k} \theta_i = \frac{- f_2(\theta_1,\theta_2)}{2} = 1 \qquad 
\sum_{i=3}^{2+k} \theta_i = \frac{- f_2(\theta_1,\theta_2)}{2} = 1 \qquad
\sum_{i=3}^{2+k} \theta_i = \frac{- f_2(\theta_1,\theta_2)}{2} = -5
\end{equation*}
with $\theta_i = \pm 1$ for all $i \in [3:2+k]$ for two, one and one common roots, respectively. 
\item[(iii)] Similarly, we computed the common roots of $\partial_{\theta_1} f_2$ and $\partial_{\theta_2} f_2$ for all remaining reflexive polygons, 
\begin{equation*}
P_{4\textup{a}}, \quad P_{4\textup{b}}, \quad P_{4\textup{c}}, \quad P_{5\textup{b}}, \quad P_{6\textup{b}}, \quad P_{6\textup{c}}, \quad P_{7\textup{a}}, \quad P_{7\textup{b}}, \quad P_{8\textup{b}}, \quad P_{8\textup{c}}. \tag{$\star$}
\end{equation*}
Since each of these roots $(\theta_1,\theta_2)$ has $- f_2 (\theta_1, \theta_2)/2 \notin \mathbb{Z}$, $\mathcal{B}^k(\star)$ has no ML degree drop. \qedhere
\end{itemize}
\end{proof}

\begin{table}[h!]
\footnotesize
\small
\begin{minipage}{0.32\textwidth}
\centering
\begin{tabular}{c c c} 
Polytope & mldeg(V) & deg(V) \\ [0.2ex] 
\hline \\ [-1ex] 
$P_3$ & 3 & 3 \\ 
$\mathcal{B}^1(P_3)$ & 6 & 6 \\
$\mathcal{B}^2(P_3)$ & 11 & 12 \\
$\mathcal{B}^3(P_3)$ & 24 & 24 \\
$\mathcal{B}^4(P_3)$ & 44 & 48 \\
$\mathcal{B}^5(P_3)$ & 96 & 96 \\
$\mathcal{B}^6(P_3)$ & 177 & 192 \\ [1ex] 

$P_{4\textup{a}}$ & 4 & 4 \\ 
$\mathcal{B}^1(P_{4\textup{a}})$ & 8 & 8 \\
$\mathcal{B}^2(P_{4\textup{a}})$ & 16 & 16 \\
$\mathcal{B}^3(P_{4\textup{a}})$ & 32 & 32 \\ 
$\mathcal{B}^4(P_{4\textup{a}})$ & 64 & 64 \\
$\mathcal{B}^5(P_{4\textup{a}})$ & 128 & 128 \\
$\mathcal{B}^6(P_{4\textup{a}})$ & 256 & 256 \\ [1ex] 

$P_{4\textup{b}}$ & 4 & 4 \\ 
$\mathcal{B}^1(P_{4\textup{b}})$ & 8 & 8 \\
$\mathcal{B}^2(P_{4\textup{b}})$ & 16 & 16 \\
$\mathcal{B}^3(P_{4\textup{b}})$ & 32 & 32 \\
$\mathcal{B}^4(P_{4\textup{b}})$ & 64 & 64 \\
$\mathcal{B}^5(P_{4\textup{b}})$ & 128 & 128 \\
$\mathcal{B}^6(P_{4\textup{b}})$ & 256 & 256 \\ [1ex] 

$P_{4\textup{c}}$ & 4 & 4 \\ 
$\mathcal{B}^1(P_{4\textup{c}})$ & 8 & 8 \\
$\mathcal{B}^2(P_{4\textup{c}})$ & 16 & 16 \\
$\mathcal{B}^3(P_{4\textup{c}})$ & 32 & 32 \\
$\mathcal{B}^4(P_{4\textup{c}})$ & 64 & 64 \\
$\mathcal{B}^5(P_{4\textup{c}})$ & 128 & 128 \\
$\mathcal{B}^6(P_{4\textup{c}})$ & 256 & 256 \\ [1ex] 

$P_{5\textup{a}}$ & 3 & 5 \\ 
$\mathcal{B}^1(P_{5\textup{a}})$ & 10 & 10 \\
$\mathcal{B}^2(P_{5\textup{a}})$ & 16 & 20 \\
$\mathcal{B}^3(P_{5\textup{a}})$ & 40 & 40 \\ 
$\mathcal{B}^4(P_{5\textup{a}})$ & 68 & 80 \\
$\mathcal{B}^5(P_{5\textup{a}})$ & 160 & 160 \\
$\mathcal{B}^6(P_{5\textup{a}})$ & 280 & 320 \\ [1ex] 

$P_{5\textup{b}}$ & 5 & 5 \\ 
$\mathcal{B}^1(P_{5\textup{b}})$ & 10 & 10 \\
$\mathcal{B}^2(P_{5\textup{b}})$ & 20 & 20 \\
$\mathcal{B}^3(P_{5\textup{b}})$ & 40 & 40 \\
$\mathcal{B}^4(P_{5\textup{b}})$ & 80 & 80 \\
$\mathcal{B}^5(P_{5\textup{b}})$ & 160 & 160 \\
$\mathcal{B}^6(P_{5\textup{b}})$ & 320 & 320 \\
\end{tabular}
\end{minipage}
\vline
\begin{minipage}{0.32\textwidth}
\centering
\begin{tabular}{c c c} 
Polytope & mldeg(V) & deg(V) \\ [0.2ex] 
\hline \\ [-1ex] 
$P_{6\textup{a}}$ & 6 & 6 \\ 
$\mathcal{B}^1(P_{6\textup{a}})$ & 10 & 12 \\
$\mathcal{B}^2(P_{6\textup{a}})$ & 24 & 24 \\
$\mathcal{B}^3(P_{6\textup{a}})$ & 42 & 48 \\
$\mathcal{B}^4(P_{6\textup{a}})$ & 96 & 96 \\
$\mathcal{B}^5(P_{6\textup{a}})$ & 172 & 192 \\
$\mathcal{B}^6(P_{6\textup{a}})$ & 384 & 384 \\ [1ex] 

$P_{6\textup{b}}$ & 6 & 6 \\ 
$\mathcal{B}^1(P_{6\textup{b}})$ & 12 & 12 \\
$\mathcal{B}^2(P_{6\textup{b}})$ & 24 & 24 \\
$\mathcal{B}^3(P_{6\textup{b}})$ & 48 & 48 \\
$\mathcal{B}^4(P_{6\textup{b}})$ & 96 & 96 \\
$\mathcal{B}^5(P_{6\textup{b}})$ & 192 & 192 \\
$\mathcal{B}^6(P_{6\textup{b}})$ & 384 & 384 \\ [1ex] 

$P_{6\textup{c}}$ & 6 & 6 \\ 
$\mathcal{B}^1(P_{6\textup{c}})$ & 12 & 12 \\
$\mathcal{B}^2(P_{6\textup{c}})$ & 24 & 24 \\
$\mathcal{B}^3(P_{6\textup{c}})$ & 48 & 48 \\
$\mathcal{B}^4(P_{6\textup{c}})$ & 96 & 96 \\
$\mathcal{B}^5(P_{6\textup{c}})$ & 192 & 192 \\
$\mathcal{B}^6(P_{6\textup{c}})$ & 384 & 384 \\ [1ex] 

$P_{6\textup{d}}$ & 6 & 6 \\ 
$\mathcal{B}^1(P_{6\textup{d}})$ & 11 & 12 \\
$\mathcal{B}^2(P_{6\textup{d}})$ & 24 & 24 \\
$\mathcal{B}^3(P_{6\textup{d}})$ & 45 & 48 \\
$\mathcal{B}^4(P_{6\textup{d}})$ & 96 & 96 \\
$\mathcal{B}^5(P_{6\textup{d}})$ & 182 & 192 \\
$\mathcal{B}^6(P_{6\textup{d}})$ & 384 & 384 \\ [1ex] 

$P_{7\textup{a}}$ & 7 & 7 \\ 
$\mathcal{B}^1(P_{7\textup{a}})$ & 14 & 14 \\
$\mathcal{B}^2(P_{7\textup{a}})$ & 28 & 28 \\
$\mathcal{B}^3(P_{7\textup{a}})$ & 56 & 56 \\
$\mathcal{B}^4(P_{7\textup{a}})$ & 112 & 112 \\
$\mathcal{B}^5(P_{7\textup{a}})$ & 224 & 224 \\
$\mathcal{B}^6(P_{7\textup{a}})$ & 448 & 448 \\ [1ex] 

 \\ 
 \\
 \\
 \\
 \\
 \\
 \\
\end{tabular}
\end{minipage}
\vline
\begin{minipage}{0.32\textwidth}
\centering
\begin{tabular}{c c c} 
Polytope & mldeg(V) & deg(V) \\ [0.2ex] 
\hline \\ [-1ex] 
$P_{7\textup{b}}$ & 7 & 7 \\ 
$\mathcal{B}^1(P_{7\textup{b}})$ & 14 & 14 \\
$\mathcal{B}^2(P_{7\textup{b}})$ & 28 & 28 \\
$\mathcal{B}^3(P_{7\textup{b}})$ & 56 & 56 \\
$\mathcal{B}^4(P_{7\textup{b}})$ & 112 & 112 \\
$\mathcal{B}^5(P_{7\textup{b}})$ & 224 & 224 \\
$\mathcal{B}^6(P_{7\textup{b}})$ & 448 & 448 \\ [1ex] 

$P_{8\textup{a}}$ & 4 & 8 \\ 
$\mathcal{B}^1(P_{8\textup{a}})$ & 16 & 16 \\
$\mathcal{B}^2(P_{8\textup{a}})$ & 24 & 32 \\
$\mathcal{B}^3(P_{8\textup{a}})$ & 64 & 64 \\
$\mathcal{B}^4(P_{8\textup{a}})$ & 104 & 128 \\
$\mathcal{B}^5(P_{8\textup{a}})$ & 256 & 256 \\
$\mathcal{B}^6(P_{8\textup{a}})$ & 432 & 512 \\ [1ex] 

$P_{8\textup{b}}$ & 8 & 8 \\ 
$\mathcal{B}^1(P_{8\textup{b}})$ & 16 & 16 \\
$\mathcal{B}^2(P_{8\textup{b}})$ & 32 & 32 \\
$\mathcal{B}^3(P_{8\textup{b}})$ & 64 & 64 \\
$\mathcal{B}^4(P_{8\textup{b}})$ & 128 & 128 \\
$\mathcal{B}^5(P_{8\textup{b}})$ & 256 & 256 \\
$\mathcal{B}^6(P_{8\textup{b}})$ & 512 & 512 \\ [1ex] 

$P_{8\textup{c}}$ & 8 & 8 \\ 
$\mathcal{B}^1(P_{8\textup{c}})$ & 16 & 16 \\
$\mathcal{B}^2(P_{8\textup{c}})$ & 32 & 32 \\
$\mathcal{B}^3(P_{8\textup{c}})$ & 64 & 64 \\
$\mathcal{B}^4(P_{8\textup{c}})$ & 128 & 128 \\
$\mathcal{B}^5(P_{8\textup{c}})$ & 256 & 256 \\
$\mathcal{B}^6(P_{8\textup{c}})$ & 512 & 512 \\ [1ex] 

$P_9$ & 9 & 9 \\ 
$\mathcal{B}^1(P_9)$ & 18 & 18 \\
$\mathcal{B}^2(P_9)$ & 36 & 36 \\
$\mathcal{B}^3(P_9)$ & 72 & 72 \\
$\mathcal{B}^4(P_9)$ & 144 & 144 \\
$\mathcal{B}^5(P_9)$ & 287 & 288 \\
$\mathcal{B}^6(P_9)$ & 576 & 576 \\ [1ex] 

 \\ 
 \\
 \\
 \\
 \\
 \\
 \\
\end{tabular}
\end{minipage}
\caption{ML degrees and degrees of $\mathcal{B}^k (P)$ for all $1 \le k \le 6$ and all reflexive polygons $P$. The ML degrees were computed using \texttt{HomotopyContinuation.jl} while the degrees were determined using Proposition \ref{proposition26}.} \label{table5}
\end{table}

\section{Reflexive Polytopes Arising from Graphs} \label{section8}

In addition to the previous constructions, reflexive polytopes can also be derived from simple graphs. A well-known example is the family of symmetric edge polytopes introduced by \citet{RootsOfEhrhartPolynomialsArisingFromGraphs}. Let $G$ be a simple undirected graph on the vertex set $[d]$ with edge set $E(G)$. The symmetric edge polytope of $G$, denoted by $\mathscr{A}_G \subseteq \mathbb{R}^d$, is the convex hull of
\begin{equation*}
A(G) \coloneqq \{ 0 \} \cup \{ \pm (e_i - e_j) \mid \{i, j\} \in E(G) \}.
\end{equation*}
See also \citep{SmoothFanoPolytopesArisingFromFiniteDirectedGraphs, InterlacingEhrhartPolynomialsOfReflexivePolytopes, ArithmeticAspectsOfSymmetricEdgePolytopes} for more information on combinatorial properties of $\mathscr{A}_G$.

\begin{example}[Permutahedron of Order 3] \upshape
Consider the permutahedron of order $3$:
\begin{equation*}
\mathcal{P}_3 \coloneqq \textup{conv} \{ \sigma = (\sigma(1), \sigma(2), \sigma(3)) \mid \sigma \in S_3 \}.
\end{equation*}
Here each permutation $\sigma$ in the symmetric group $S_3$ written in tuple notation is interpreted as a vector in $\mathbb{R}^3$. It is a translation of the symmetric edge polytope $\mathscr{A}_G$ arising from
\begin{figure}[h!]
\centering
\begin{tikzpicture}
\coordinate[label=left:$1$] (A) at (0,0);
\coordinate[label=right:$2$] (B) at (2,0);
\coordinate[label=above:$3$] (C) at (1,1.5);
\coordinate[label=left:$G\; {=}$] (D) at (-0.75,1);
\draw (A) -- (B);
\draw (A) -- (C);
\draw (B) -- (C);
\fill (A) circle (2pt);
\fill (B) circle (2pt);
\fill (C) circle (2pt);
\end{tikzpicture}.
\end{figure}

\noindent Both descriptions yield a two-dimensional polytope embedded in a  three-dimensional space. According to \citet{AZonotopeAssociatedWithGraphicalDegreeSequences}, $\textup{vol}(\mathcal{P}_d) = d^{d-2}$. Since $\textup{vol}(\mathcal{P}_3) = 3$ and the $f$-vector is $(6,6)$, $\mathscr{A}_G$ is isomorphic to the reflexive polygon of type 6a, which has no ML degree drop.
\end{example}

A well-known complete bipartite graph is the star
\begin{figure}[H]
\centering
\begin{tikzpicture}
\coordinate[label=left:$1$] (A) at (0,0);
\coordinate[label=right:$2$] (B) at (0,1.5);
\coordinate[label=above:$3$] (C) at (1,1);
\coordinate[label=right:$4$] (D) at (1.5,0);
\coordinate[label=right:$5$] (E) at (1,-1);
\coordinate[label=below:$6$] (F) at (0,-1.5);
\coordinate[label=below:$7$] (G) at (-1,-1);
\coordinate[label=above:$d$] (H) at (-1,1);
\coordinate[label=left:$\cdot$] (I) at (-1.2,0);
\coordinate[label=left:$\cdot$] (J) at (-1.18,0.2);
\coordinate[label=left:$\cdot$] (K) at (-1.18,-0.2);
\coordinate[label=left:$K_{1,d-1}\; {=}$] (L) at (-2,0);
\draw (A) -- (B);
\draw (A) -- (C);
\draw (A) -- (D);
\draw (A) -- (E);
\draw (A) -- (F);
\draw (A) -- (G);
\draw (A) -- (H);
\fill (A) circle (2pt);
\fill (B) circle (2pt);
\fill (C) circle (2pt);
\fill (D) circle (2pt);
\fill (E) circle (2pt);
\fill (F) circle (2pt);
\fill (G) circle (2pt);
\fill (H) circle (2pt);
\end{tikzpicture}.
\end{figure}

\begin{proposition}
The log-linear models defined by $\mathscr{A}_{K_{1,d-1}}$ and $\cross_{d-1}$ are isomorphic.
\end{proposition}

\begin{proof}
The symmetric edge polytope $\mathscr{A}_{K_{1,d-1}}$ is a $(d-1)$-dimensional polytope embedded in a $d$-dimensional space. It defines the design matrix
\begin{equation*}
A
 = \begin{bmatrix}
0 & \pm (e_1-e_2) & \pm (e_1 - e_3) & \ldots & \pm (e_1 - e_d)
\end{bmatrix} \in \mathbb{Z}^{d\times 2(d-1)+1}
\end{equation*}
of rank $d-1$. The first row can be represented as a linear combination of the other rows. Removing the first row gives
\begin{equation*}
\begin{bmatrix}
0 & \pm e_1 & \pm e_2 & \ldots & \pm e_{d-1}
\end{bmatrix} \in \mathbb{Z}^{(d-1)\times 2(d-1)+1},
\end{equation*}
which is the design matrix of $\cross_{d-1}$.
\end{proof}

\begin{corollary}
For $d>1$,
\begin{equation*}
\textup{mldeg}(\mathscr{A}_{K_{1,d-1}}) = \textup{deg}(\mathscr{A}_{K_{1,d-1}}) = 2^{d-1}.
\end{equation*}
\end{corollary}

Another construction of reflexive polytopes based on graphs was introduced by \citet{ReflexivePolytopesArisingFromBipartiteGraphsWithGammaPositivityAssociatedToInteriorPolynomials}. Let $\mathscr{B}_G \subseteq \mathbb{R}^d$ denote the convex hull of 
\begin{equation*}
B(G) \coloneqq \{ 0, \pm e_1, \ldots, \pm e_d\} \cup \{ \pm e_i \pm e_j \mid \{ i, j \} \in E(G) \}.
\end{equation*}
Then $\mathscr{B}_G$ is reflexive if and only if $G$ is bipartite \citep[Theorem 0.1]{ReflexivePolytopesArisingFromBipartiteGraphsWithGammaPositivityAssociatedToInteriorPolynomials}. For $E(G) = \emptyset$, $\mathscr{B}_G$ is the $d$-dimensional cross polytope. If only one node is added to a bipartite graph $G$ without any further edges, this corresponds to construction $\mathcal{B}$ from Section \ref{section5}.

\begin{example}[2-Cube] \upshape
Consider
\begin{figure}[h!]
\centering
\begin{tikzpicture}
\coordinate[label=left:$1$] (A) at (0,0);
\coordinate[label=right:$2$] (B) at (2,0);
\coordinate[label=left:$G\; {=}$] (D) at (-0.75,0);
\draw (A) -- (B);
\fill (A) circle (2pt);
\fill (B) circle (2pt);
\end{tikzpicture}.
\end{figure}

\noindent Then $\mathscr{B}_G$ corresponds to $P_{8\textup{a}}$. The polytope $\mathcal{B}(P_{8\textup{a}})$ arises from
\begin{figure}[H]
\centering
\begin{tikzpicture}
\coordinate[label=left:$1$] (A) at (0,1.5);
\coordinate[label=right:$2$] (B) at (2,1.5);
\coordinate[label=below:$3$] (C) at (1,0.5);
\coordinate[label=left:$G'\; {=}$] (D) at (-0.75,1);
\draw (A) -- (B);
\fill (A) circle (2pt);
\fill (B) circle (2pt);
\fill (C) circle (2pt);
\end{tikzpicture}.
\end{figure}

\noindent The associated ML degree and the degree of the toric variety of $\mathscr{B}_G$ and $\mathscr{B}_{G'}$ were examined in detail in the previous sections.
\end{example}

Since $K_{1,d-1}$ is bipartite for all $d$, we have that $\mathscr{B}_{K_{1,d-1}}$ is reflexive. We finish by giving the degree and ML degree of the corresponding model.

\begin{proposition}
For $d \ge 2$,
\begin{equation*}
\textup{deg}(\mathscr{B}_{K_{1,d-1}}) = d \cdot 2^d \quad \text{and} \quad  \textup{mldeg}(\mathscr{B}_{K_{1,d-1}}) = 2^d. 
\end{equation*}
\end{proposition}

\begin{proof}
The polytope $\mathscr{B}_{K_{1,d-1}}$ is the Cartesian product $\cross_1 \times \cross_{d-1}$. Therefore,
\begin{equation*}
\textup{vol}(\mathscr{B}_{K_{1,d-1}}) = 2 \cdot \textup{vol}(\cross_{d-1}) = \frac{2^d}{(d-1)!}. 
\end{equation*}
The expression for the ML degree follows by applying Corollary \ref{cor:product} and Theorem \ref{theorem24}:
\begin{equation*}
\textup{mldeg}(\mathscr{B}_{K_{1,d-1}}) = \textup{mldeg}(\cross_1) \cdot \textup{mldeg}(\cross_{d-1}) = 2 \cdot 2^{d-1} = 2^d. \qedhere
\end{equation*}
\end{proof}

\addcontentsline{toc}{section}{References}
\bibliographystyle{plainnat}
\bibliography{bib}

\vspace{0.25cm}

\noindent{\bf Authors' addresses:}
\smallskip
\small 

\noindent Carlos Am\'{e}ndola,
Technische Universit\"at Berlin
\hfill {\tt amendola@math.tu-berlin.de}

\noindent Janike Oldekop,
Technische Universit\"at Berlin
\hfill {\tt oldekop@math.tu-berlin.de}

\end{document}